\documentclass[11pt]{amsart}

\usepackage{amssymb,amsmath,amsfonts,
,enumerate}
\usepackage[all,cmtip]{xy}

\newtheorem{proposition}{Proposition}[section]
\newtheorem{definition}[proposition]{Definition}
\newtheorem{lemma}[proposition]{Lemma}
\newtheorem{theorem}[proposition]{Theorem}

\newtheorem{corollary}[proposition]{Corollary}
\newtheorem{MainTheorem}{Theorem}

\theoremstyle{definition}
\newtheorem{remark}[proposition]{Remark}
\newtheorem{example}[proposition]{Example}

\DeclareMathOperator{\vol}{vol}

\DeclareMathOperator{\Val}{Val}

\DeclareMathOperator{\Gr}{Gr}

\DeclareMathOperator{\SL}{SL}

\DeclareMathOperator{\SO}{SO}
\renewcommand{\O}{\mathrm{O}}

\DeclareMathOperator{\nc}{nc}

\DeclareMathOperator{\pd}{pd}

\DeclareMathOperator{\Area}{Area}
\DeclareMathOperator{\glob}{glob}
\DeclareMathOperator{\Sym}{Sym}

\newcommand{\R}{\mathbb{R}}
\newcommand{\C}{\mathbb{C}}

\newcommand{\id}{\mathrm{id}}
\newcommand{\sgn}{\mathrm{sgn}}

\DeclareMathOperator{\V}{\mathrm{V}}

\DeclareMathOperator{\Flag}{Flag}
\DeclareMathOperator{\FlagArea}{FlagArea}

\renewcommand{\V}{V}

\newcommand{\largewedge}{\mbox{\Large $\wedge$}}
\newcommand{\tr}{\mathrm{tr}}

\date{\today}



\title{Additive kinematic formulas for flag area measures}
\author{Judit Abardia-Ev\'equoz} 

\author{Andreas Bernig} 
\address{Institut f\"ur Mathematik, Goethe-Universit\"at Frankfurt am Main, 
Robert-Mayer-Str. 10, 60054 Frankfurt, Germany}
\email{abardia@math.uni-frankfurt.de}
\email{bernig@math.uni-frankfurt.de}

\begin{document}

\date{\today}

\begin{abstract}
We show the existence of additive kinematic formulas for general flag area measures, which generalizes a recent result by Wannerer. Building on previous work by the second named author, we introduce an algebraic framework to compute these formulas explicitly. This is carried out in detail in the case of the incomplete flag manifold consisting of all $(p+1)$-planes containing a unit vector. 
\end{abstract}

\thanks{J. A. was supported by DFG grant AB 584/1-2. A. B. was supported by DFG grant BE 2484/5-2.}
\maketitle

\tableofcontents

\section{Introduction}

\subsection{Global and local kinematic formulas}

Let $V$ be a Euclidean vector space of dimension $n$ and unit sphere $S^{n-1}$. Let $\mathcal{K}(V)$ denote the space of compact convex subsets in $V$. A \emph{valuation} is a finitely additive map on $\mathcal{K}(V)$, that is, a map $\mu:\mathcal{K}(V) \to A$, where $A$ is any abelian semigroup, and such that
\begin{displaymath}
\mu(K \cup L)+\mu(K \cap L)=\mu(K)+\mu(L), 
\end{displaymath} 
whenever $K,L,K \cup L \in \mathcal{K}(V)$. If $A$ is a topological semigroup, then continuity of $\mu$ is understood with respect to the Hausdorff topology on $\mathcal{K}(V)$.

The \emph{intrinsic volumes} $\mu_0,\ldots,\mu_n$ are examples of real valued valuations. They are continuous, translation invariant, and invariant under rotations. Hadwiger has shown that the space $\Val^{\SO(n)}$ of continuous, translation and rotation invariant valuations is spanned by the intrinsic volumes. If $G$ is a subgroup of $\SO(n)$, then $\Val^G$ (the space of continuous, translation invariant and $G$-invariant valuations) is finite-dimensional if and only if $G$ acts transitively on the unit sphere. 

The classification of connected compact Lie groups $G$ acting transitively on some
sphere is a topological problem which was solved by Montgomery-Samelson \cite{montgomery_samelson43} and Borel \cite{borel49}. There are six infinite lists
\begin{displaymath} 
 \mathrm{SO}(n), \mathrm{U}(n/2), \mathrm{SU}(n/2), \mathrm{Sp}(n/4), \mathrm{Sp}(n/4) \cdot \mathrm{U}(1), \mathrm{Sp}(n/4)\cdot \mathrm{Sp}(1),
\end{displaymath}
and three exceptional groups
\begin{displaymath} 
 \mathrm{G}_2, \mathrm{Spin}(7), \mathrm{Spin}(9). 
\end{displaymath}
For the sake of brevity, we will call a group from this list a \emph{transitive group}. The computation of the dimension as well as an explicit geometric description of a basis of $\Val^G$ in each of these cases is a challenging problem and many results have been obtained recently by various authors \cite{alesker_mcullenconj01, bernig_sun09, bernig_g2, bernig_qig, bernig_fu_hig,  bernig_solanes, bernig_solanes17, bernig_voide}.  

Many important constructions in convex geometry can be interpreted as valuations taking values in some abelian semigroup $A$ other than $\R$. We mention briefly two cases. If $A=(\mathcal{K}(V),+)$, where $+$ stands for the Minkowski sum, then $A$-valued valuations are called \emph{Minkowski valuations} \cite{haberl10,ludwig_2005,schuster_wannerer18,schuster_wannerer}. If $A=\mathrm{Sym}^*V$, then $A$-valued valuations are called \emph{tensor valuations} \cite{bernig_hug,  boroczky_domokos_solanes, hug_schneider_localtensor, hug_schneider_schuster_a, hug_schneider_schuster_b,hug_weis17, Kiderlen_Vedel_Jensen}.

Federer's curvature measures $C_0,\ldots,C_n$ are valuations with values in the space of signed measures on $V$. We refer to Schneider \cite{schneider78} for a classification result of these curvature measures. Fu \cite{fu90} and Bernig-Fu-Solanes \cite{bernig_fu_solanes} studied smooth curvature measures in a broader context. For a transitive group $G$, there are local kinematic formulas for translation and $G$-invariant smooth curvature measures. The case $G=\SO(n)$ was known by Federer, whereas the hermitian case $G=\mathrm U(n/2)$ was described completely in \cite{bernig_fu_hig,bernig_fu_solanes, bernig_fu_solanes_proceedings}. 

The classical surface area measures $S_0,\ldots,S_{n-1}$ are valuations with values in the space of signed measures on the unit sphere. Additive kinematic formulas for surface area measures have been shown in \cite{schneider75}. Smooth area measures which are equivariant with respect to a transitive group $G$ were introduced by Wannerer \cite{wannerer_area_measures,wannerer_unitary_module}, who proved the existence of local additive kinematic formulas and established such formulas in the hermitian case. Since these results are very influential for the present work, let us state them explicitly.

\begin{theorem}
Let $G$ be a transitive group. Then the space $\Area^{G}$ of smooth, $G$-invariant area measures is finite-dimensional. If $\Phi_1,\ldots,\Phi_N$ is a basis of $\Area^{G}$, then there are local additive kinematic formulas 
\begin{displaymath}
\int_G \Phi_i(K +gL,\kappa \cap g\lambda)dg=\sum_{k,l} c_{k,l}^i \Phi_k(K,\kappa) \Phi_l(L,\lambda),
\end{displaymath}   
where $K,L \in \mathcal K(V)$ and where $\kappa$ and $\lambda$ are Borel subsets of $S^{n-1}$.
\end{theorem}

The linear map 
\begin{displaymath}
A:\Area^{G} \to \Area^{G} \otimes \Area^{G}, \ \Phi_i \mapsto \sum c_{k,l}^i \Phi_k \otimes \Phi_l
\end{displaymath}
is a cocommutative, coassociative coproduct. The transposed map $A^*:\Area^{G,*} \otimes \Area^{G,*} \to \Area^{G,*}$ thus provides $\Area^{G,*}$ with the structure of a commutative associative algebra.

In the case $G=\SO(n)$, this algebra is isomorphic to $\R[t]/{\langle t^n\rangle}$. The case $G=\mathrm{U}(n/2)$ is more involved. Using results from hermitian integral geometry \cite{bernig_fu_hig} and from the theory of tensor valuations \cite{bernig_hug}, Wannerer could write down this algebra in an explicit way.
\begin{theorem}
The algebra $\Area^{\mathrm{U}(n/2),*}$ is isomorphic to 
\begin{displaymath}
\R[s,t,v]/I_n,
\end{displaymath}
where $I_n$ is the ideal generated by 
\begin{displaymath}
f_{n/2+1}(s,t),f_{n/2+2}(s,t),p_{n/2}(s,t)-q_{n/2-1}(s,t)v, v^2
\end{displaymath}
with 
\begin{align*}
\log(1+tx+sx^2) & = \sum_{k=0}^\infty f_k(s,t)x^k\\
\frac{1}{1+tx+sx^2} & = \sum_{k=0}^\infty p_k(s,t)x^k\\
-\frac{1}{(1+tx+sx^2)^2} & = \sum_{k=0}^\infty q_k(s,t)x^k.
\end{align*}
\end{theorem}

Although it is possible to write down the local additive kinematic formulas for $G=\mathrm U(n/2)$ using this theorem, the result is not as explicit as one would like. In \cite{bernig_dual_area_measures}, the second named author has shown that the algebra structure on $\Area^{G,*}$ is induced from the algebra structure of a larger (infinite-dimensional) space $\Area^{*,sm}$ of \emph{smooth dual area measures} and that the product in this algebra can be computed in some easy, algorithmic way. This has led to very explicit local additive kinematic formulas.

\subsection{Results of the present paper}

Let $V$ be an $n$-dimensional Euclidean vector space with unit sphere $S^{n-1}$. Let $G$ be a subgroup of $\O(n)$ acting transitively on the sphere $S^{n-1}$ and let $\overline G:=G \ltimes V$. 

Let $H$ be a closed subgroup of $G \cap \O(n-1)$, where $\O(n-1) \subset \O(n)$ is the stabilizer of the first basis vector. Then $M:=G/H$ is a $G$-homogeneous manifold. 

A \emph{smooth flag area measure} is a translation invariant valuation with values in the space of signed measures on $M=G/H$ which is given by a certain smooth differential form on $V \times M$. We denote by $\Area_{G/H}$ the space of smooth flag area measures and by $\Area^G_{G/H}$ the subspace of smooth flag area measures which are equivariant with respect to the action, i.e., 
\begin{displaymath}
 \Phi(gK,g\kappa)=\Phi(K,\kappa), K \in \mathcal{K}(V), \kappa \in \mathcal{B}(M), g \in G.
\end{displaymath}

We refer to Section \ref{sec_existence} for the complete definition. 

Our first two main theorems generalize \cite[Corollary 3.1]{wannerer_area_measures} and \cite[Theorem 2]{bernig_dual_area_measures}, where the special case $M=S^{n-1}$ was considered. 

\begin{MainTheorem}[Existence of local additive kinematic formulas]\label{thm_existence1}
The space $\Area^G_{G/H}$ is finite-dimensional if and only if $G$ acts transitively on the unit sphere. In this case, if $\{\Phi_1,\dots,\Phi_N\}$ is a basis of $\Area^G_{G/H}$, then there exist local additive kinematic formulas
\begin{displaymath}
 \int_G \Phi_i(K+gL,\kappa\cap g\lambda)dg=\sum_{k,l=1}^N c^i_{k,l}\Phi_k(K,\kappa)\Phi_l(L,\lambda),
\end{displaymath}
where $\kappa,\lambda$ are Borel subsets of $M$. 
\end{MainTheorem}

Examples:
\begin{enumerate}
\item[(i)] Let $G$ be a subgroup of $\O(n)$ which acts transitively on $S^{n-1}$. Let $H=G \cap \O(n-1)$ be the stabilizer of the action. Then $M=S^{n-1}$ and the flag area measures in this case are called \emph{area measures}. The kinematic formulas in the case $G=\O(n)$ are classical \cite{schneider75}. For more general $G$, Wannerer has shown the existence of kinematic formulas \cite[Theorem 2.1]{wannerer_area_measures} and our proof follows his arguments. Explicit formulas in the hermitian case $G=\mathrm U(n)$ are contained in \cite{bernig_dual_area_measures, solanes, wannerer_area_measures}. 
\item[(ii)] Let $G=\SO(n), H=S(\O(p) \times \O(q))$,  $p+q=n-1$. Then $M=\Flag_{1,p+1}$ is the incomplete flag manifold consisting of pairs $(v,E)$ with $v \in S^{n-1}, E \in \Gr_{p+1}(V)$. Some elements in the space $\FlagArea^{(p),\SO(n)}:=\Area^G_{G/H}$ were constructed using a Steiner type formula in \cite{hug_tuerk_weil}, see also \cite{hinderer_hug_weil}. A complete description of this space was given in \cite{abardia_bernig_dann}.
 \item[(iii)] If $G \subset \SO(n)$ and $H=\{1\}$, we also call the elements in $\Area^{G}_{G/\{1\}}$ \emph{rotation measures}. They will be studied in Section \ref{sec_rotation_measures}.
\end{enumerate}

The local additive kinematic formulas can be encoded by the map
\begin{displaymath}
 A:\Area^G_{G/H} \to \Area^G_{G/H} \otimes \Area^G_{G/H},\ \Phi_i \mapsto \sum_{k,l} c_{k,l}^i \Phi_k \otimes \Phi_l,
\end{displaymath}
which is a cocommutative, coassociative coproduct. Alternatively, the dual space $(\Area^{G,*}_{G/H},A^*)$ is a commutative, associative algebra.

We give an explicit construction and classification of rotation measures. Consider a smooth compact convex body and $x \in \partial K$. Let $\nu(x)$ be the outer normal vector at $x$. Given $g \in \SO(n)$ with $ge_1=\nu(x)$, the vectors $ge_2,\ldots,ge_n$ span $T_x\partial K$. The shape operator is the self-adjoint linear map $S_x:=d\nu_x:T_x\partial K \to T_x\partial K$.

\begin{MainTheorem} \label{mainthm_rotation_measures}
The space of rotation measures of degree $k$ is of dimension  
\begin{displaymath}
\dim \Area^{\SO(n),*}_{k,\SO(n)/\{1\}}=\frac{1}{n}\binom{n}{k}\binom{n}{k+1}.
\end{displaymath}
For each $0 \leq k \leq n-1$ and for each ordered subsets $I,J \subset \{2,\ldots,n\}$ with $|I|=|J|=k$ there exists a unique rotation measure $S_{I,J} \in \Area^{\SO(n)}_{\SO(n)/\{1\},k}$ such that for every compact convex body with smooth boundary 
\begin{equation} \label{eq_rotation_on_smooth}
S_{I,J}(K,f)=\int_{\partial K}  \int_{\substack{g \in \SO(n)\\ ge_1=\nu(x)}} f(g) \det(\pi_{J^\perp} \circ S_x|_{V_I^\perp}:V_I^\perp \to V_J^\perp) dg \ d\mathcal H^{n-1}(x).
\end{equation}
Here $V_I=\mathrm{span}\{ge_i,i \in I\}, V_J:=\mathrm{span}\{ge_j,j \in J\}$ are oriented $k$-dimensional subspaces of $T_x\partial K$ and $\pi_{J^\perp}:T_x\partial K \to V_J^\perp$ is the orthogonal projection.

There are linear relations among these rotation measures: $S_{I,J}$ is antisymmetric in $I$ and antisymmetric in $J$. Moreover, given $I'=(i_1',\ldots,i_{k-1}') \subset \{2,\ldots,n\},J'=(j_1',\ldots,j'_{k+1}) \subset \{2,\ldots,n\}$ we have 
\begin{equation} \label{eq_relation_rotation_measures}
\sum_{l=1}^{k+1} (-1)^l S_{\{i_1',\ldots,i_{k-1}',j_{l}'\},\{j_1',\ldots,j_{l-1}',j_{l+1}',\ldots,j_{k+1}'\}}=0.
\end{equation}
The space $\Area^{\SO(n)}_{\SO(n)/\{1\},k}$ of rotation measures is isomorphic to the quotient of the vector space spanned by the $S_{I,J}$ by the relations \eqref{eq_relation_rotation_measures}. 
\end{MainTheorem}

The next theorem describes algebraically the kinematic formulas for $\Area^{G}_{G/H}$ for any closed subgroup $H \subset \SO(n-1)$. 

Let us write $S_{i,j}:=S_{\{i\},\{j\}}$ in the case $k=1$. Then \eqref{eq_relation_rotation_measures} translates to $S_{i,j}= S_{j,i}$. Hence the element $\frac{S^*_{i,j}+S^*_{j,i}}{2}$ belongs to $\Area^{\SO(n)*}_{\SO(n)/\{1\},1}$.

\begin{MainTheorem} \label{mainthm_algebra}
The map $x_{i,j} \mapsto - \vol \SO(n) \cdot \frac{S^*_{i,j}+S^*_{j,i}}{2}$ induces a graded algebra morphism 
\begin{displaymath}
\R[X]/I \cong \Area^{\SO(n),*}_{\SO(n)/\{1\}},
\end{displaymath}
where $X:=(x_{i,j})_{2\leq i,j\leq n}$ and where $I \subset \R[X] $ is the ideal generated by 
\begin{displaymath}
x_{i,j}x_{k,l}+x_{i,k}x_{j,l}+x_{i,l}x_{j,k},\quad x_{i,j}-x_{j,i}.
\end{displaymath}
If $H \subset \SO(n-1)$ is a closed subgroup, then 
\begin{displaymath}
\Area^{\SO(n),*}_{\SO(n)/H} \cong \left(\R[X]/I\right)^H,
\end{displaymath}
with respect to the action $h^*X=h^tXh$. 
\end{MainTheorem}

In the important case of the incomplete flag manifold consisting of pairs $(v,E)$ with $\dim E=p+1, v \in E$, we write down the algebra structure more explicitly.
\begin{proposition} \label{prop_algebra_flag_case}
Let $p+q=n-1$, $G=\SO(n), H=S(\O(p)\times \O(q)), M=G/H=\Flag_{1,p+1}$. If $p \neq \frac{n-1}{2}$, then there is a graded isomorphism  
\begin{displaymath}
\FlagArea^{(p),\SO(n),*} \cong \R[x,y]/\langle x^{p+1},y^{q+1}\rangle,
\end{displaymath}
where $\deg x=\deg y=1$.
If $p =q= \frac{n-1}{2}$, then there is a graded isomorphism  
\begin{displaymath}
\FlagArea^{(p),\SO(n),*} \cong \R[x,y,u]/\left\langle x^{p+1},y^{p+1},ux,uy,u^2-(-1)^p\frac{p+1}{2^{2p}}x^py^p \right\rangle,
\end{displaymath}
where $\deg u=\frac{n-1}{2}$.
\end{proposition}

A basis of the space $\FlagArea^{(p),\SO(n)}$ consists of flag area measures $S_k^{(p),i}$, where $0 \leq k \leq n-1, 0 \leq i \leq m_k$ (together with an additional flag area measure $\tilde S_{\frac{n-1}{2}}^{\left(\frac{n-1}{2}\right)}$ if $p=q=\frac{n-1}{2}$), see \cite{abardia_bernig_dann}. The algebra structure given in the previous corollary translates into explicit kinematic formulas for these flag area measures. 

Let $\omega_n$ denote the volume of the $(n-1)$-dimensional unit sphere and let
\begin{displaymath}
c_{n,k,p,i}:=\binom{n-1}{k}^{-1}\binom{m_{k}}{i}^{-1}\binom{|k-q|+m_{k}}{i}^{-1}\binom{n-1}{i},
\end{displaymath}
as in \cite{abardia_bernig_dann}. For given $p,q,k$, we write $m_{k}:=\min\{p,q,k,n-k-1\}$, $m'_{k}:=\min\{p,k\}$ and define 
\begin{align*}
C^{k,i}_{j,b,c} & := (-1)^{b+c+m'_{j}+m'_{k-j}} \frac{c_{n,k,p,i}}{c_{n,j,p,b}c_{n,k-j,p,c}} \sum_{t=m'_{k}-m_{k}}^{m'_{k}-i} (-1)^t \binom{m'_{k}-t}{i} \cdot \\
& \quad \cdot \sum_{s=\max\{m'_{j}-b,t-m_{k-j}\}}^{\min\{m_{j},t-m'_{k-j}+c\}} \binom{b}{m'_{j}-s}\binom{c}{m'_{k-j}-t+s} \cdot \\
& \quad \cdot\binom{q}{j-s}^{-1}\binom{p}{s}^{-1}\binom{q-k+t-s+j}{j-s}\binom{p-t+s}{s}.
\end{align*}
We do not know whether this expression can be simplified any further.

\begin{MainTheorem}[Local additive kinematic formulas for flag area measures] \label{thm_kin_S}
Let $0\leq p,k\leq n-1$,  and $0\leq i\leq m_{k}$. Then, 
\begin{enumerate}
\item[(i)] If $(k,p) \neq \left(n-1,\frac{n-1}{2}\right)$, then 
\begin{displaymath}
A_{1,p+1}^{\SO(n)}\left(S_k^{(p),i}\right)=\frac{1}{\omega_n}\sum_{j=0}^k\sum_{b=0}^{m_{j}}\sum_{c=0}^{m_{k-j}}
C^{k,i}_{j,b,c}S_j^{(p),b}\otimes S_{k-j}^{(p),c},
\end{displaymath}
\item[(ii)] If $(k,p) = \left(n-1,\frac{n-1}{2}\right)$, then  
\begin{align*}
A_{1,\frac{n+1}{2}}^{\SO(n)}\left(S_{n-1}^{\left(\frac{n-1}{2}\right),0}\right) & =\frac{1}{\omega_n}\sum_{j=0}^{n-1}\sum_{b,c=0}^{m_{j}} 
C^{n-1,0}_{j,b,c}S_j^{\left(\frac{n-1}{2}\right),b}\otimes S_{k-j}^{\left(\frac{n-1}{2}\right),c} \\
& \quad + \frac{(-1)^{\frac{n-1}{2}} (n+1)}{2^{n}\omega_n} \tilde S_{\frac{n-1}{2}}^{\left(\frac{n-1}{2}\right)} \otimes \tilde S_{\frac{n-1}{2}}^{\left(\frac{n-1}{2}\right)},
\end{align*}
\item[(iii)] If $(k,p) = \left(\frac{n-1}{2},\frac{n-1}{2}\right)$, then 
\begin{displaymath}
A^{\SO(n)}_{1,\frac{n+1}{2}}\left(\tilde S_{\frac{n-1}{2}}^{\left(\frac{n-1}{2}\right)}\right)=\frac{1}{\omega_n}\left(\tilde S_{\frac{n-1}{2}}^{\left(\frac{n-1}{2}\right)} \otimes S_0^{\left(\frac{n-1}{2}\right),0}+S_0^{\left(\frac{n-1}{2}\right),0} \otimes \tilde S_{\frac{n-1}{2}}^{\left(\frac{n-1}{2}\right)}\right).
\end{displaymath}
\end{enumerate}
\end{MainTheorem}

\smallskip
\subsection*{Acknowledgments}
We thank Thomas Wannerer for useful comments on a first draft of this paper.

\section{Existence of kinematic formulas for smooth flag area measures}\label{sec_existence}

In this section we will introduce smooth area measures. Our definition will be justified by the existence of kinematic formulas for smooth area measures that will be shown in this section. 

\subsection{Fiber integration}

We first collect some definitions and results from the theory of fiber bundles that will be used in the following.
\begin{definition}\label{def_push-forward}
Let $(E,\pi,B,F)$ be a fiber bundle, with $B$ and $E$ oriented manifolds and $F$ compact with $\dim F=r$. The \emph{fiber integration} or \emph{push-forward} of a form $\eta\in\Omega^{d+r}(E)$ is the form $\pi_*\eta \in \Omega^{d}(B)$ such that
\begin{displaymath}
\int_B \omega \wedge \pi_* \eta=  \int_{E} \pi^{*}\omega \wedge \eta,
\end{displaymath}
for every differential form $\omega\in\Omega^*(B)$ with compact support. 
\end{definition}

We then have the projection formula
\begin{equation} \label{eq_projection_formula}
\pi_*(\pi^*\omega \wedge \eta)=\omega \wedge \pi_*\eta.
\end{equation}

We note that for the above definition of fiber integration, we follow the sign convention in \cite{alvarez_fernandes}, as in \cite{abardia_bernig_dann, wannerer_area_measures}. For another sign convention see, e.g., \cite{berline_getzler_vergne}.

\begin{lemma}[{\cite[Equation (7)]{wannerer_area_measures}}] \label{lemma_push_forward_inverse}
Let $(E,\pi,B,F)$ be a fiber bundle, with $B$ orientable and $E$ oriented with the local product orientation. If $N\subset B$ is a compact and oriented submanifold with $\dim N=n$ and $\pi^{-1}(N)\subset E$ has the local product orientation, then, for every $\omega\in\Omega^{n+r}(E)$ with fiber-compact support,
$$\int_N \pi_*\omega=\int_{\pi^{-1}(N)}\omega.$$
\end{lemma}

\begin{lemma}[{\cite[4.3.2, 4.3.8]{federer_book}}] \label{lemma_sard}
Let $f:M\to N$ be a smooth and surjective map with $M,N$ compact and oriented smooth manifolds with $\dim M=m$ and $\dim N=n$. Then $f^{-1}(y)$ is an orientable smooth submanifold for almost every $y\in N$. 
Let the orientation of $M$ be given by a smooth $m$-vector field $\xi$, and let the orientation of $N$ be given by the smooth form $dy \in \Omega^n(N)$. If the orientation of $f^{-1}(y)$ is given by the smooth $(m-n)$-vector field $\xi \llcorner f^*dy$ and if $\mu$ denotes the measure given by $[[Y]] \llcorner dy$, then
\begin{displaymath}
\int_N\varphi(y)\left(\int_{f^{-1}(y)}\omega\right)d\mu(y)=\int_Mf^{*}(\varphi\wedge dy)\wedge\omega,
\end{displaymath}
for every continuous function $\varphi:N\to\R$ and every $\omega\in\Omega^{m-n}(M)$.
\end{lemma}

\begin{lemma}[\cite{fu90}] \label{lemma_fu} 
Let $(E,\pi,M,F)$ and $(E',\pi',M',F)$ be oriented fiber bundles with the same fiber $F$. Assume that $\overline f:E'\to E$ is a bundle map covering a smooth map $f:M'\to M$ and that there exists an open cover $\{U\}$ of $M$ with local trivializations $\varphi:\pi^{-1}(U)\to U\times F$ and $\varphi':\pi^{-1}(f^{-1}(U))\to f^{-1}(U)\times F$ compatible with the orientations of the bundle such that  $\varphi \circ\overline f=(f\times\id_F) \circ \varphi'$, then 
$$f^*\circ \pi_*=\pi'_*\circ\overline f^*.$$
\end{lemma}

\subsection{Smooth flag area measures}

Let us introduce our main object of study. Let $e_1$ be the first standard vector in $\R^n$ and let $\O(n-1)$ be its stabilizer. Let $M:=G/H$ be a homogeneous space, where $G$ is a closed subgroup of $\O(n)$ and where $H$ is a closed subgroup of $G \cap \O(n-1)$. Let $\Pi:V \times M \to SV, (x,gH) \mapsto (x,ge_1)$, which is a fiber bundle. We let $r$ denote the dimension of the fiber and $m:=n-1+r$. 

\begin{definition}
A translation invariant valuation $\Phi$ with values in the space of signed measures on $M$ is called a \emph{smooth flag area measure} if there is a translation invariant differential form $\omega \in \Omega^m(V \times M)$ such that for every $f \in C^\infty(M)$ we have 
\begin{displaymath}
\int_{M} f d\Phi(K,\cdot)=\int_{\nc(K)} \Pi_*(f \omega),
\end{displaymath}
where $\nc(K)$ denotes the normal cycle of $K$. The space of smooth flag area measures is denoted by $\Area_{G/H}$, and $\Area^G_{G/H}$ denotes the subspace of smooth flag area measures equivariant under the action of $G$ given by $(g\Phi)(K,f)=\Phi(g^{-1}K,g^*f)$.
\end{definition}

\begin{definition}
Let $H_1 \subset H_2 \subset G$ be subgroups. Let $\hat \Pi: V \times G/H_1 \to V \times G/H_2$ 
denote the projection map. The \emph{globalization map} 
\begin{displaymath}
\glob: \Area_{G/H_1} \to \Area_{G/H_2}
\end{displaymath}
is defined by 
\begin{displaymath}
\int_{G/H_2} f d\glob \Phi(K,\cdot)=\int_{G/H_1} \hat\Pi^* f d\Phi(K,\cdot), \quad f \in C^\infty(G/H_2).
\end{displaymath}
\end{definition}

\begin{lemma} \label{lemma_globalization}
Let $\hat \Phi \in \Area^G_{G/H_1}$ be represented by $\hat \omega$. Then $\glob \hat \Phi$ is represented by $\hat \Pi_*\hat \omega$.
\end{lemma}

\proof
Let $\Pi:V \times G/H_2 \to SV$ be the projection map. For $f \in C^\infty(G/H_2)$ and $K \in \mathcal{K}(V)$, we obtain, by using \eqref{eq_projection_formula},
\begin{align*}
\int_{G/H_2} f d \glob \hat \Phi(K,\cdot) & =\int_{G/H_1} \hat \Pi^*f d \hat \Phi(K,\cdot)\\
& =\int_{\nc(K)} (\Pi \circ \hat \Pi)_*(\hat \Pi^*f \cdot \hat \omega)\\
& =\int_{\nc(K)} \Pi_* (f \cdot \hat \Pi_*\hat \omega).
\end{align*}
Hence $\glob \hat \Phi$ is represented by the form $\hat \Pi_* \hat \omega$.
\endproof

The space $\Omega^l(V \times M)^{\tr}$ of translation invariant forms admits a filtration as follows. For $(x,gH) \in V \times M, 0\leq j\leq l$, we define
 \begin{align*}
  \mathfrak{F}^{l,j}_{x,gH} & := \{\phi \in \largewedge^l T^*_{(x,gH)}(V \times M)\,: \\
  & \quad \quad \forall v_1,\ldots,v_j \in T_{(x,gH)} \Pi^{-1}(\Pi(x,gH)), \ \phi(v_1,\ldots,v_j,-)=0 \},\\
  \mathfrak{F}^{l,j} & :=\{\omega \in \Omega^l(V \times M)^{\tr} :\, 
   \forall (x,gH) \in V \times M, \ \omega|_{(x,gH)} \in \mathfrak{F}^{l,j}_{x,gH} \},\\
  \mathfrak{F}^{\bullet,j} & := \bigoplus_l \mathfrak{F}^{l,j}.
 \end{align*}
 Then 
 \begin{displaymath}
  \Omega^l(V \times M)^{\tr} = \mathfrak{F}^{l,r+1} \supset \mathfrak{F}^{l,r} \supset \ldots \supset \mathfrak{F}^{l,0}=\{0\}.
 \end{displaymath} 

 \begin{proposition} \label{prop_kernel}
 A form $\tau \in \Omega^m(V \times M)^{\tr}$ induces the trivial flag area measure if and only if 
  \begin{displaymath}
   \tau \in \langle \Pi^*\alpha, \Pi^* d\alpha,\mathfrak{F}^{m,r}\rangle.
  \end{displaymath}
 \end{proposition}
 
\proof
In the special case $G=\O(n), H=\O(p) \times \O(q)$, this was shown in \cite[Theorem 2.3]{abardia_bernig_dann}. The proof can be easily adapted to the general case. 
\endproof

\subsection{Kinematic formulas}

If $G$ is transitive, then the space of smooth flag area measures is a quotient of the finite-dimensional space $\Omega^m(V \times M)^{\overline G}$ and hence finite-dimensional itself. If $G$ is not transitive, then $\Val^G$ is not finite-dimensional. Since we have a surjective map $\glob:\Area^G_{G/H} \to \Val^G$, $\dim \Area^G_{G/H}=\infty$ as well in this case. 

By using the definition of smooth flag area measures, we can restate the remaining part of Theorem~\ref{thm_existence1} as follows. 

\begin{theorem}\label{existence2}
Let $G$ be transitive. 
Let $\omega \in \Omega^{m}(V \times M)^{\overline G}$ and let $\{\omega_1,\dots,\omega_N\}$ be a basis of $\Omega^m(V \times M)^{\overline G}$. Then there exist constants $c_{k,l}$ such that
\begin{displaymath}
\int_{G} \int_{\nc(K+gL)} \Pi_*(\varphi \cdot (g^{-1})^*\psi \cdot \omega) dg=\sum_{k,l} c_{k,l} \int_{\nc(K)}\Pi_*(\varphi \omega_k) \int_{\nc(L)}\Pi_*(\psi\omega_l),
\end{displaymath}
where $\varphi,\psi \in C^\infty(M)$. 
\end{theorem}

\begin{proof}[Proof of Theorem~\ref{existence2}]
We follow the ideas of Fu \cite{fu90} and Wannerer \cite{wannerer_area_measures}. We first assume that the convex bodies $K$ and $L$ have smooth boundaries.
 
We define 
\begin{displaymath}
E:=\{(g,(x_i,p_i)_{i=1,2,3}) \in G \times (V \times M)^3\,:\,
x_1+gx_2=x_3, p_1=gp_2=p_3\},
\end{displaymath}
and the maps
\begin{align*}
p:E & \to (V \times M) \times (V \times M)\\
(g,x_1,p_1,x_2,p_2,x_3,p_3) & \mapsto ((x_1,p_1),(x_2,p_2)),\\
q:E & \to G \times V \times M\\
(g,x_1,p_1,x_2,p_2,x_3,p_3) & \mapsto (g,x_3,p_3).
\end{align*}

We define $E'$ together with maps $p':E' \to SV \times SV, q':E' \to G \times SV$ in an analogous way, using $S^{n-1}$ instead of $M$. We then have  a map 
\begin{align*}
\Pi':E & \to E', \\
(g,(x_i,p_i)) & \mapsto (g,\Pi(x_i,p_i)). 
\end{align*}

Let $q_1:E \to G, q_1':E' \to G$ be the projections on the first factors. Let $\Pi_1,\Pi_2:(V \times M) \times (V \times M) \to M$ be the projection maps onto the $M$-factors. Let $l_i:G \times V \times M \to M$ be defined by $l_1(g,x,p):=p, l_2(g,x,p):=g^{-1}p$. Then, $\Pi_i \circ p=l_i \circ q,i=1,2$ and we have the diagram 
\begin{displaymath}
\xymatrix{
& M & \\
(V \times M) \times (V \times M) \ar[d]^{\Pi \times \Pi} \ar[ur]^{\Pi_i} & E  \ar[l]_-p \ar[r]^-q \ar[d]^{\Pi'} & G \times V \times M \ar[d]^{\mathrm{id} \times \Pi} \ar[ul]_{l_i} \\
SV \times SV & E' \ar[d]_{q_1'} \ar[l]_-{p'} \ar[r]^-{q'} & G \times SV \\
& G & 
}
\end{displaymath}

Given $\bar h \in \overline G$, we let $h$ be the part in $G$, that is, $h(x)=\overline h(x)-\overline h(0)$. 
The group $\overline G \times \overline G$ acts on $V \times M \times V \times M, E$, and $G \times V \times M$ as follows:
\begin{align*}
(\bar h,\bar k) \cdot (g,x_1,p_1,x_2,p_2,x_3,p_3) & := (hgk^{-1},\bar hx_1,h p_1,\bar kx_2, kp_2,\bar h x_3,h p_3),\\
(\bar h,\bar k) \cdot (x_1,p_1,x_2,p_2) & := (\bar h x_1,h p_1,\bar k x_2, kp_2),\\
(\bar h,\bar k) \cdot (g,x_3,p_3) & := (hgk^{-1},\bar h x_3, h p_3). 
\end{align*}
The actions on $SV \times SV, E'$, and $G \times SV$ are defined analogously. Then the maps in the above diagram are $\bar G \times \bar G$-equivariant. 

Using Lemma \ref{lemma_fu} we obtain that $p_* q^*(dg \wedge \omega)$ is a $\bar G \times \bar G$-invariant form, i.e., an element of $\Omega^*(V \times M)^{\bar G} \times \Omega^*(V \times M)^{\bar G}$. The component of bidegree $(m,m)$ can be written as $\sum_{k,l} c_{k,l} \omega_k \wedge \omega_l$.

The normal cycle $K+gL$ is given by 
\begin{displaymath}
\nc(K+gL)= q'_*\left[(q_1')^{-1}(g) \cap (p')^* (\nc(K) \times \nc(L))\right].
\end{displaymath}

Applying the pull-back $\Pi^*$, Lemma \ref{lemma_fu} yields 
\begin{align*}
\Pi^* \nc(K+gL) & = \Pi^* (q')_* \left[(q_1')^{-1}(g) \cap (p')^* (\nc(K) \times \nc(L))\right]\\
& = q_* (\Pi')^* \left[(q_1')^{-1}(g) \cap (p')^* (\nc(K) \times \nc(L))\right]\\
& = q_* \left[q_1^{-1}(g) \cap (\Pi')^* (p')^* (\nc(K) \times \nc(L))\right]\\
& = q_*\left[q_1^{-1}(g) \cap p^* (\Pi^*\nc(K) \times \Pi^*\nc(L))\right].
\end{align*}

We set $F:=p^{-1} (\Pi^{-1}\nc(K) \times \Pi^{-1}\nc(L)) \subset E$. 

Using the previous computation and Lemmas \ref{lemma_push_forward_inverse} and \ref{lemma_sard} (applied to $(q_1)|_F:F \to G$) we deduce that 
\begin{align*}
\int_G\int_{\nc(K+gL)}\Pi_* & (\varphi \cdot (g^{-1})^* \psi \cdot \tau)dg =\int_G\int_{\Pi^{-1} \nc(K+gL)} \varphi \cdot (g^{-1})^*\psi \cdot \tau dg \\
&= \int_G \int_{q_1^{-1}(g) \cap F} q^*(\varphi \cdot (g^{-1})^* \psi) \tau dg\\
&= \int_F q^*(\varphi \cdot (g^{-1})^* \psi \wedge \tau) \wedge dg \\
&= \int_{\Pi^{-1}(\nc(K)) \times \Pi^{-1}(\nc(L))} p_*\left[q^*(\varphi \cdot (g^{-1})^*\psi \cdot \tau) \wedge dg\right] \\
&= \int_{\Pi^{-1}(\nc(K)) \times \Pi^{-1}(\nc(L))} p_*\left[q^* l_1^*\varphi \cdot q^*l_2^*\psi \cdot q^*\tau \wedge dg\right] \\
&= \int_{\Pi^{-1}(\nc(K)) \times \Pi^{-1}(\nc(L))} p_*\left[p^* \Pi_1^*\varphi \cdot p^*\Pi_2^*\psi \cdot q^*\tau \wedge dg\right] \\
& =\int_{\Pi^{-1}(\nc(K) )\times \Pi^{-1}(\nc(L))} \Pi_1^*\varphi \cdot \Pi_2^*\psi \cdot p_*(q^*\tau \wedge dg)\\
& =\int_{\Pi^{-1}(\nc(K))\times\Pi^{-1}(\nc(L))} \Pi_1^*\varphi \cdot \Pi_2^*\psi \cdot \left(\sum_{k,l}c_{k,l} \tau_k \wedge \tau_l\right)\\
&=\sum_{k,l} c_{k,l}\int_{\Pi^{-1}(\nc(K))}\varphi \tau_k \int_{\Pi^{-1}(\nc(L))}\psi\tau_l\\
&=\sum_{k,l}c_{k,l}\int_{\nc(K)}\Pi_*(\varphi\tau_k)\int_{\nc(L)}\Pi_*(\psi\tau_l).
\end{align*}

The general case of not necessarily smooth convex bodies follows by approximation as in \cite{wannerer_area_measures}. 
\end{proof}

\section{Dual flag area measures}\label{sec_dual}

In this section we introduce the notion of smooth dual flag area measure, which generalizes the notion of smooth dual area measure from \cite{bernig_dual_area_measures}. Similarly to the case of dual area measures, we define a convolution product on the space of smooth dual area measures. 

The space $\Omega^m(V \times M)^{\tr}$ of translation invariant forms is endowed with the usual Fr\'echet topology of uniform convergence on compact subsets of all partial derivatives and there is a surjection $\Omega^m(V\times M)^{\tr} \to \Area_{G/H}$. We endow the latter space with the quotient topology and denote by $\Area_{G/H}^*$ the dual space to $\Area_{G/H}$.

Analogously to the case of area measures, we define a convolution product on a subspace of $\Omega^n(\V \times M)^{\tr}$.

First, we introduce an operator on the space of differential forms on $V \times M$ which will play the analogous role, and is defined analogously, to $*_1$ from \cite{bernig_fu06} and \cite{bernig_dual_area_measures}. We denote this operator again by $*_1$.
  
\begin{definition} 
The linear operator 
\begin{displaymath}
*_1:\Omega^{*}(V\times M)^{\tr} \to\Omega^*(V \times M)^{\tr}
\end{displaymath}
is given by
\begin{displaymath}
*_1(\pi_1^*\tau_1\wedge \pi_2^* \tau_2)=(-1)^{\binom{n-k}{2}} \pi_1^*(*\tau_1)\wedge \pi_2^*\tau_2,
\end{displaymath}
where $\tau_1\in \bigwedge^k\V$, $\tau_2\in\Omega^*(M)$ and $*:\bigwedge^k \V\to\bigwedge^{n-k}V$ denotes the Hodge star operator. Here $\pi_1:V\times M \to V$ and $\pi_2:V \times M \to M$ denote the natural projections. 
 \end{definition}
 
\begin{lemma}\label{lemma_def_J_conv}
The space 
\begin{align*}
\mathcal{J}^{n,tr} & =\mathcal{J}^{n,tr}(V \times M)\\
& :=\{\tau\in\Omega^n(\V \times M)^{\tr}\,:\,\tau \in \mathfrak{F}^{n,1}, \Pi^*\alpha \wedge \tau=0, \Pi^* d\alpha \wedge\tau=0\}
\end{align*}
is closed under the operation
\begin{equation}\label{conv_flag_dual}
\tau *\tau':=*_1^{-1}(*_1\tau\wedge *_1\tau').
\end{equation}
\end{lemma}

\begin{proof}
Since $*_1$ does not affect the fiber of $\Pi$, we have $\tau \in \mathfrak{F}^{\bullet,j}$ if and only if $*_1\tau \in \mathfrak{F}^{\bullet,j}$. Moreover, if $\rho_i \in \mathfrak F^{\bullet,j_i}, i=1,2$, then $\rho_1 \wedge \rho_2 \in \mathfrak{F} ^{\bullet,j_1+j_2-1}$. Hence, if $\tau,\tau'\in\mathfrak{F}^{n,1}$, we have $\tau *\tau'\in\mathfrak{F}^{n,1}$.  

The other two conditions can be proved as in the case $p=0$ (see \cite{bernig_dual_area_measures}). 
\end{proof}

\begin{definition}
Let $\tau \in \Omega^{2n-1+r}(V \times M)^{\tr}$. Then $(\pi_1)_*\tau \in \Omega^n(V)^{\tr}$ is a multiple of the Lebesgue measure. We denote this multiple by $\int \tau$.
\end{definition}

\begin{definition} \label{def_smooth_dual_area}
 A dual flag area measure $L \in \Area_{G/H}^*$ is called smooth if there exists  $\tau \in \mathcal{J}^{n,tr}$ such that 
\begin{equation} \label{eq_paing_dual_area}
  \langle L,\Phi\rangle = \int \tau \wedge \omega,
 \end{equation}
whenever $\omega \in \Omega^{n-1+r}(V \times M)^{\tr}$ represents $\Phi \in \Area_{G/H}$. The space of smooth dual flag area measures is denoted by $\Area_{G/H}^{*,sm}$.
\end{definition}

\begin{lemma} \label{lemma_glob_dual}
Let $H_1 \subset H_2 \subset G$ be subgroups and $\hat \Pi: V \times G/H_1 \to V \times G/H_2$ the projection map.
Let $L \in \Area_{G/H_2}^{G*}$ be represented by the form $\tau \in \mathcal J^{n,tr}(V \times G/H_2)$. Then $\glob^*L \in \Area_{G/H_1}^{G*}$ is represented by $\hat \Pi^*\tau \in \mathcal J^{n,tr}(V \times G/H_1)$. In particular, $\glob^*$ is injective.
\end{lemma}

\proof
By Lemma \ref{lemma_globalization} and Definition \ref{def_push-forward}, we have 
\begin{align*}
\langle \glob^*L,\hat \Phi\rangle &= \langle L,\glob \hat \Phi\rangle = \int \tau \wedge \hat \Pi_*\hat \omega = \int \hat\Pi^*\tau \wedge \hat\omega. 
\end{align*}

\endproof

\begin{definition} \label{def_convolution_dual}
 Let $L_1,L_2 \in \Area_{G/H}^{*,sm}$ be represented by forms $\tau_1,\tau_2 \in \mathcal{J}^{n,tr}$. Then we define $L_1 * L_2 \in \Area_{G/H}^{*,sm}$ as the smooth dual area measure represented by $\tau_1 * \tau_2=*_1^{-1} (*_1\tau_1 \wedge *_1\tau_2) \in \mathcal{J}^{n,tr}$.
\end{definition}

\begin{theorem} \label{thm_ftaig}
Let $G \subset \O(n)$ be a closed subgroup acting transitively on the unit sphere and $H \subset G \cap \O(n-1)$ be a closed subgroup. Then the following diagram commutes
\begin{displaymath}
 \xymatrix{
 \Area_{G/H}^{*,sm} \otimes \Area_{G/H}^{*,sm} \ar[d]^{q_G \otimes q_G} \ar[r]^-{*} & \Area_{G/H}^{*,sm} \ar[d]^{q_G}\\
 \Area_{G/H}^{G*} \otimes \Area_{G/H}^{G*} \ar[r]^-{A^*} & \Area_{G/H}^{G*}
 }
\end{displaymath}
Here $q_G$ is the map transposed to the inclusion $\Area^G_{G/H} \hookrightarrow \Area_{G/H}$.
\end{theorem}

We need some preparation before proving the theorem.

\begin{proposition}
Let $G$ be as in the theorem and let $(\Gamma,\rho)$ be a finite-dimen\-sional $G$-module. The group $G$ acts on $\Val \otimes \Gamma$ by $(g\mu)(K):=\rho(g)\mu(g^{-1}K)$. Then the space $(\Val \otimes \Gamma)^G$ is finite-dimensional and contained in $(\Val^{sm} \otimes \Gamma)^G$.
\end{proposition}

\proof
Let $\mu \in (\Val \otimes \Gamma)^G$. By Alesker's irreducibility theorem \cite{alesker_mcullenconj01} we may approximate $\mu$ by a sequence $\mu_i \in \Val^{sm} \otimes \Gamma$. Averaging with respect to Haar measure on $G$ we find an approximating sequence $\tilde \mu_i \in (\Val^{sm} \otimes \Gamma)^G$. But the latter space is finite-dimensional, since it is a quotient of the space of translation- and $G$-invariant, $\Gamma$ valued smooth differential forms, which is obviously finite-dimensional. It follows that $\mu$ belongs to $(\Val^{sm} \otimes \Gamma)^G$. 
\endproof

\begin{proposition}[Kinematic formulas for tensor valuations]
Let $(\Gamma_i,\rho_i)$ be finite-dimensional $G$-modules. If $\mu_1,\ldots,\mu_k$ is a basis of $(\Val \otimes \Gamma_1)^G$ and $\phi_1,\ldots,\phi_l$ is a basis of $(\Val \otimes \Gamma_2)^G$, then for every $\tau \in (\Val \otimes \Gamma_1 \otimes \Gamma_2)^G$ there are constants $c_{ij}^\tau$ such that 
 \begin{displaymath}
  \int_G (\mathrm{id} \otimes \rho_2(g^{-1})) \tau(K+gL)dg=\sum_{i,j} c_{i,j}^\tau \mu_i(K) \otimes \phi_j(L).
 \end{displaymath}
We thus obtain a cocommutative coassociative coproduct 
\begin{displaymath}
a^{\Gamma_1,\Gamma_2}: (\Val \otimes \Gamma_1 \otimes \Gamma_2)^G  \to (\Val \otimes \Gamma_1)^G \otimes (\Val \otimes \Gamma_2)^G.
\end{displaymath}
\end{proposition}

\proof
This follows from the usual Hadwiger argument, compare \cite[Section 3.2]{bernig_hug} for a similar situation.  
\endproof

\begin{definition}
Let $(\Gamma_1,\rho_1),(\Gamma_2,\rho_2)$ be finite-dimensional $G$-modules. We define the \emph{convolution} $c:(\Val \otimes \Gamma_1)^G \otimes (\Val \otimes \Gamma_2)^G \to (\Val \otimes \Gamma_1 \otimes \Gamma_2)^G$ componentwise, i.e., $(\mu_1 \otimes f_1) * (\mu_2 \otimes f_2):=(\mu_1 * \mu_2) \otimes f_1 \otimes f_2$. 
\end{definition}

\begin{proposition}
Let $(\Gamma,\rho)$ be a finite-dimensional $G$-module. The bilinear map 
\begin{displaymath}
c: (\Val_k \otimes \Gamma)^G \otimes (\Val_{n-k} \otimes \Gamma^*)^G \to (\Val_0 \otimes \Gamma \otimes \Gamma^*)^G \cong \Val_0^G \cong \R
\end{displaymath} 
gives rise to an isomorphism 
$\widehat{\pd}:(\Val_k \otimes \Gamma)^G \to (\Val_{n-k}^* \otimes \Gamma)^G$, which is called \emph{Poincar\'e duality}. 
\end{proposition}

We remark that there is another version of Poincar\'e duality which uses the Alesker product instead of the convolution. Following \cite{wannerer_area_measures} we put a hat to distinguish between the two dualities. 

\begin{proposition}
Let $(\Gamma_1,\rho_1),(\Gamma_2,\rho_2)$ be finite-dimensional $G$-modules. Then the following diagram commutes
 \begin{displaymath}
 \xymatrix{
  (\Val \otimes \Gamma_1 \otimes \Gamma_2)^G \ar[r]^-{a^{\Gamma_1,\Gamma_2}} \ar[d]^{\widehat{\pd}} & (\Val \otimes \Gamma_1)^G \otimes (\Val \otimes \Gamma_2)^G \ar[d]^{\widehat{\pd} \otimes \widehat{\pd}}\\
 (\Val^* \otimes \Gamma_1 \otimes \Gamma_2)^G \ar[r]^-{c_G^*} & (\Val^* \otimes \Gamma_1)^G \otimes (\Val^* \otimes \Gamma_2)^G
 }
\end{displaymath}
\end{proposition}

\proof
The proof is analogous to \cite[Theorem 4.6]{wannerer_area_measures}. 
\endproof

\begin{definition}
Let $\Gamma \subset C^\infty(G)$ be a finite-dimensional $G$-submodule. The map $M^\Gamma:\Area^G_{G/\{1\}} \to (\Val \otimes \Gamma^*)^G$ such that 
\begin{displaymath}
 \langle M^\Gamma \Phi(K),f\rangle=\Phi(K,f), \quad \Phi \in \Area^G_{G/\{1\}}, f \in \Gamma,
\end{displaymath}
is called \emph{moment map}.
\end{definition}

In the particular case $\Gamma=\Sym^lV$, the moment map was already used in \cite{wannerer_area_measures} and \cite{bernig_dual_area_measures}.

\begin{lemma} \label{lemma_kernel_moment_maps}
\begin{enumerate}
 \item[(i)]  \begin{displaymath}
  \bigcap_\Gamma \ker M^\Gamma=\{0\},
 \end{displaymath}
where the intersection is over all finite-dimensional subrepresentations of $C^\infty(G)$.
\item[(ii)] The linear span of the images of $(M^\Gamma)^*:(\Val^* \otimes \Gamma)^G \to \Area^{G*}_{G/\{1\}}$, where $\Gamma$ ranges over all finite-dimensional subrepresentations of $C^\infty(G)$, equals $\Area^{G*}_{G/\{1\}}$. 
\end{enumerate}
\end{lemma}

\proof
\begin{enumerate}
 \item[(i)] Suppose that $M^\Gamma \Phi=0$ for all $\Gamma$. This means that $\Phi(K,f)=0$ for all $f \in \Gamma$ and all $\Gamma$. Since $\bigoplus_\Gamma \Gamma$ is a dense subspace of $C^\infty(G)$ by the theorem of Peter-Weyl \cite{broecker_tomdieck}, this implies that $\Phi(K,f)=0$ for every $f \in C^\infty(G)$, hence $\Phi=0$.
\item[(ii)] Let $\mathcal H \subset \Area^{G*}_{G/\{1\}}$ be the linear span of the images of $(M^\Gamma)^*$. If $\mathcal H \neq \Area^{G*}_{G/\{1\}}$, then we find some $0 \neq \Phi \in \Area^G_{G/\{1\}}$ with $\langle L,\Phi\rangle=0$ for all $L \in \mathcal H$. But this implies that $\langle \mu,M^\Gamma \Phi\rangle=0$ for all $\mu \in (\Val^* \otimes \Gamma)^G$. By Poincar\'e duality, this means that $M^\Gamma \Phi=0$ for all $\Gamma$, hence $\Phi=0$ by (i). This is a contradiction. Hence $\mathcal H=\Area^{G*}_{G/\{1\}}$.  
\end{enumerate}
\endproof

Let $D:\Omega^{n-1}(SV) \to \Omega^n(SV)$ denote the Rumin operator \cite{rumin94}. 

\begin{lemma}
Let $\Gamma \subset C^\infty(G)$ be a finite-dimensional $G$-submodule. If $\mu \in (\Val_{n-k} \otimes \Gamma)^G$ is represented by $\eta=\sum_i \eta_i \otimes f_i, \eta_i \in \Omega^{n-1}(SV)^{\tr},f_i \in \Gamma$, then $(\widehat{\pd} \circ M^\Gamma)^* \mu \in \Area^{G*}_{G/\{1\}}$ is represented by 
\begin{displaymath}
 (-1)^k\Pi^*D\eta:=(-1)^k \sum_i f_i\Pi^*D\eta_i \in \Omega^n(M).
\end{displaymath}
\end{lemma}

\proof
Let $\Phi \in \Area^G_{G/\{1\}}$ be given by a form $\omega \in \Omega^m(V \times M)$ of bidegree $(k,m-k)$. Using \cite[Proposition 4.2]{wannerer_area_measures} we find 
\begin{align*}
\langle (\widehat{\pd} \circ M^\Gamma)^* \mu,\Phi\rangle & = \langle \widehat{\pd} \circ M^\Gamma \Phi,\mu\rangle \\
& = (-1)^k \sum_i \int  D\eta_i \wedge \Pi_*(f_i \omega)  \\
& = (-1)^k \int  \sum_i f_i\Pi^*D\eta_i \wedge \omega.
\end{align*}
\endproof

\begin{lemma}
 The additive kinematic formulas are compatible in the following sense. Let $\Gamma_1,\Gamma_2 \subset C^\infty(G)$ be finite-dimensional submodules and let $\Gamma_1 \cdot \Gamma_2$ be the (finite-dimensional) module generated by all $f_1 \cdot f_2, f_1 \in \Gamma_1,f_2 \in \Gamma_2$. Let $q:\Gamma_1 \otimes \Gamma_2 \to \Gamma_1 \cdot \Gamma_2, f_1 \otimes f_2 \mapsto f_1 \cdot f_2$ be the natural projection, Then the following diagram commutes 
 \begin{displaymath}
  \xymatrix{
 \Area^G_{G/\{1\}} \ar[r]^-A \ar[d]^-{M^{\Gamma_1 \cdot \Gamma_2}}  & \Area^G_{G/\{1\}} \otimes \Area^G_{G/\{1\}} \ar[dd]^{M^{\Gamma_1} \otimes M^{\Gamma_2}} \\
 (\Val \otimes (\Gamma_1 \cdot \Gamma_2)^*)^G   \ar[d]^-{\mathrm{id} \otimes q^*}& \\
 (\Val \otimes \Gamma_1^* \otimes \Gamma_2^*)^G \ar[r]^-{a^{\Gamma_1^*,\Gamma_2^*}}  & (\Val \otimes \Gamma_1^*)^G \otimes (\Val \otimes \Gamma_2^*)^G
 }
 \end{displaymath}
\end{lemma}

\proof
Let $f_i \in \Gamma_i, i=1,2$.
We compute
\begin{align*}
\bigg\langle & \left[a^{\Gamma_1^*,\Gamma_2^*} \circ (\mathrm{id} \otimes q^*) \circ M^{\Gamma_1 \cdot \Gamma_2} \Phi\right](K,L),f_1 \otimes f_2\bigg\rangle \\
& = \int_G \left\langle (\mathrm{id} \otimes \rho_2^*(g^{-1})) \left(q^*  M^{\Gamma_1 \cdot \Gamma_2}\Phi (K+gL)\right), f_1 \otimes f_2\right\rangle dg\\
& = \int_G \langle M^{\Gamma_1 \cdot \Gamma_2}\Phi(K+gL), f_1 \cdot (g^{-1})^*f_2\rangle dg\\
& = \int_G \Phi(K+gL, f_1 \cdot (g^{-1})^*f_2) dg\\
& = A(\Phi)(K,L,f_1,f_2)\\
& = \langle (M^{\Gamma_1} \otimes M^{\Gamma_2}) \circ A(\Phi)(K,L),f_1 \otimes f_2\rangle.
\end{align*}
\endproof

\proof[Proof of Theorem \ref{thm_ftaig}]
We look first at rotation measures. 

Let $\Gamma_1,\Gamma_2 \subset C^\infty(G)$ be two finite-dimensional $G$-submodules. Dualizing the above diagram yields the commutative diagram
\begin{displaymath}
 \xymatrix{
 \Area^{G*}_{G/\{1\}}  &  \Area^{G*}_{G/\{1\}} \otimes \Area^{G*}_{G/\{1\}} \ar[l]_-{A^*}\\
 (\Val \otimes \Gamma_1 \otimes \Gamma_2)^G  \ar[u]^{(\widehat{\pd} \circ (\mathrm{id} \otimes q^*) \circ M^{\Gamma_1 \cdot \Gamma_2})^*} & (\Val \otimes \Gamma_1)^G \otimes (\Val \otimes \Gamma_2)^G \ar[l]_-{c_G}  \ar[u]_{(\widehat{\pd} \circ M^{\Gamma_1})^* \otimes (\widehat{\pd} \circ M^{\Gamma_2})^*}
 }
\end{displaymath}

Let $\mu_a \in (\Val_{n-k_a} \otimes \Gamma_a)^G, a=1,2$ be represented by $\eta^a=\sum_i \eta^a_i \otimes f_i^a$ of bidegree $(n-k_a,k_a-1)$. Then $\mu_1 * \mu_2$ is represented by some $\eta \in \Omega^{n-1}(SV) \otimes \Gamma_1 \otimes \Gamma_2$ such that $D\eta=\sum_{i,j} *_1^{-1}(*_1 D\eta^1_i \wedge *_1 D\eta^2_j) \otimes f_{i}^1 \otimes f_j^2$, which has bidegree $(n-k_1-k_2,k_1+k_2-1)$. It follows that 
\begin{displaymath}
(-1)^{k_1+k_2} \Pi^*(D\eta)=*_1^{-1}(*_1(-1)^{k_1}\Pi^*D\eta^1 \wedge *_1(-1)^{k_2}\Pi^*D\eta^2).
\end{displaymath}
This means that if $L_1 \in \Area^{G*}_{G/\{1\}}$ is in the image of $(\widehat{\pd} \circ M^{\Gamma_1})^*$ and $L_2$ is in the image of $(\widehat{\pd} \circ M^{\Gamma_2})^*$, then the formula to compute $L_1 * L_2$ is correct. Thus, by Lemma \ref{lemma_kernel_moment_maps} the formula holds for all $L_1,L_2 \in \Area^{G*}_{G/\{1\}}$.

In the general case, we have a commutative diagram
\begin{displaymath}
\xymatrix{\Area^{G*}_{G/\{1\}}  & \Area^{G*}_{G/\{1\}}  \otimes \Area^{G*}_{G/\{1\}} \ar[l]_-{*}\\
\Area^{G*}_{G/H} \ar[u]_{\glob^*} & \Area^{G*}_{G/H}  \otimes \Area^{G*}_{G/H} \ar[l]_-{*} \ar[u]_{\glob^* \otimes \glob^*}
}
\end{displaymath}
Since $\glob^*$ is injective, the statement follows from Lemma \ref{lemma_glob_dual} and the fact that $\hat \Pi^*$ commutes with $*_1$.
\endproof

\section{Rotation measures} \label{sec_rotation_measures}

In this section we consider rotation measures, i.e., the case $G:=\SO(n),H:=\{1\}$ and prove Theorems \ref{mainthm_rotation_measures} and \ref{mainthm_algebra}. 

\subsection{Classification of rotation measures}

\proof[Proof of Theorem \ref{mainthm_rotation_measures}]
In the first part of the proof, we follow \cite[Section 3]{abardia_bernig_dann}.

The Lie algebras of $G,\overline{G}$ will be denoted by $\bar{\mathfrak g}, \mathfrak g$. The dimension of the fiber of $\Pi:V \times G \to SV$ is given by $r:=\binom{n-1}{2}$. As above we set $m:=n-1+r$. By definition, a smooth flag area measure $\Phi \in \Area_{\SO(n)/\{1\}}^{\SO(n)}$ of degree $k$ is represented by a translation invariant differential form $\eta \in \Omega^{k,m-k}(V \times G)$. Since $\Phi$ is $G$-invariant, we may assume by averaging over $G$ that $\eta$ is $G$-invariant, i.e., $\eta \in \Omega^m(\bar G)^{\bar G}$. 

Let $\sigma_i,i=1,\ldots,n, \omega_{ij}, 1 \leq i,j \leq n$ denote the components of the Maurer-Cartan form of $\bar G$. Then $\sigma_i,1 \leq i \leq n, \omega_{i,j}, 1\leq i < j \leq n$ span $\bar{\mathfrak g}^*$. We let $X_i, 1 \leq i \leq n, X_{ij}, 1 \leq i<j \leq n$ denote the dual basis of $\mathfrak g$. We let $V_0$ be the span of $X_1$; $V_\sigma$ the span of $X_i, 2 \leq i \leq n$; $V_\omega$ be the span of the $X_{1j}, 2 \leq j \leq n$; and $U$ the span of the $X_{ij}, 2 \leq i<j\leq n$. Schematically, the Lie algebra looks as follows:
\begin{displaymath}
\bar{\mathfrak g}=\left(\begin{matrix} 0 & & \\
V_0 & 0 & \\
V_\sigma & V_\omega & U\end{matrix}\right).
\end{displaymath}

Since $\Pi^*\alpha=\sigma_1$, the quotient of the space of $\bar G$-invariant forms of bidegree $(k,m-k)$ by multiplies of $\Pi^*\alpha$ is the space  
\begin{displaymath}
\Lambda^{k,m-k}(V_\sigma^* \oplus V_\omega^* \oplus U^*) \cong \Lambda^kV_\sigma^* \otimes \Lambda^{m-k}(V_\omega^* \oplus U^*) \cong \Lambda^kV_\sigma^* \otimes \bigoplus_{i=0}^{m-k} \Lambda^{m-k-i}V_\omega^* \otimes \Lambda^i U^*.
\end{displaymath}

If $\eta$ belongs to the sum of terms with $i<\dim U^*=r$, then $\eta \in \mathfrak F^{m,r}$ and hence induces the trivial flag area measure. We thus obtain that 
\begin{displaymath}
\Omega^{k,m-k}(\bar G)^{\bar G}/\langle \Pi^*\alpha,\mathfrak F^{m,r}\rangle \cong \Lambda^k V_\sigma^* \otimes \Lambda^{n-1-k} V_\omega^*.
\end{displaymath}

Next, by Proposition \ref{prop_kernel}, we have to factor out multiples of $\Pi^*d\alpha$, so that 
\begin{align*}
\Area_{\SO(n)/\{1\},k}^{\SO(n)} & \cong \Omega^{k,m-k}(\bar G)^{\bar G}/\langle \Pi^*\alpha,\Pi^*d\alpha,\mathfrak F^{m,r}\rangle \\
& \cong (\Lambda^k V_\sigma^* \otimes \Lambda^{n-1-k} V_\omega^*)/(\Lambda^{k-1} V_\sigma^* \otimes \Lambda^{n-2-k} V_\omega^*).
\end{align*}

In particular, 
\begin{displaymath}
\dim \Area_{\SO(n)/\{1\},k}^{\SO(n)}=\binom{n-1}{k}^2-\binom{n-1}{k-1} \binom{n-1}{k+1}=\frac{1}{n}\binom{n}{k}\binom{n}{k+1}.
\end{displaymath}

We now construct the rotation measures $S_{I,J}$. Uniqueness follows from the fact that smooth convex bodies are dense in the space of all compact convex bodies. Let us prove existence. Let $I=(i_1,\ldots,i_k),J=(j_1,\ldots,j_k)$ and set $I^c=\{2,\ldots,n\} \setminus I$, ordered in such a way that $\sgn(2,\ldots,n)=\sgn(i_1,\ldots,i_k,i_1^c,\ldots,i_{n-k-1}^c)$ and similar for $J^c$.

Define $S_{I,J}$ by the differential form 
\begin{displaymath}
\omega_{I,J}:=\sigma_{i_1} \wedge \ldots \wedge \sigma_{i_k} \wedge \omega_{j_1^c,1} \wedge \ldots \wedge \omega_{j^c_{n-k-1},1} \wedge \rho \in \Omega^{k,m-k}(V \times \SO(n)),
\end{displaymath}
where $\rho$ is the volume form of the fiber of the map $\SO(n) \to S^{n-1},g \mapsto ge_1$. 

Let $K$ be a smooth compact convex body with outer unit normal $\nu:\partial K \to S^{n-1}$. Fix $x \in \partial K$ and $g \in \SO(n)$ with $ge_1=\nu(x)$. Then $ge_2,\ldots,ge_n$ form a positive orthonormal basis of $T_x\partial K$. The vectors $w_l:=(ge_l,S_x(ge_l)) \in T_x\partial K \times T_x\partial K,l=2,\ldots,n$ span $T_{(x,\nu(x))}\nc(K)$. 

We have for $2 \leq i,j \leq n$ 
\begin{displaymath}
\sigma_i(w_l)=\delta_{il}, \quad \omega_{j,1}(w_l)=\langle ge_j,S(ge_l)\rangle.
\end{displaymath}

With $\Sigma_k$ denoting the permutation group of $k$ elements, we thus have
\begin{align*}
\omega_{I,J}(w_2,\ldots,w_n) & = \omega_{I,J}(w_{i_1},\ldots,w_{i_k},w_{i^c_1},\ldots,w_{i^c_{n-k-1}})\\
& = \sum_{\pi \in \Sigma_{n-k-1}} \!\!\!\! \sgn(\pi) \langle ge_{j^c_1},S(ge_{i^c_{\pi_1}})\rangle \cdot \ldots \cdot \langle ge_{j^c_{n-k-1}},S(ge_{i^c_{\pi_{n-k-1}}})\rangle \\ 
& = \det(\pi_{J^\perp} \circ S_x|_{V_I^\perp}:V_I^\perp \to V_J^\perp),
\end{align*}
from which \eqref{eq_rotation_on_smooth} follows.

By the definition of rotation measures (see Definition \ref{def_smooth_dual_area}) and by Proposition \ref{prop_kernel}, every rotation measure is a linear combination of the $S_{I,J}$. The relations \eqref{eq_relation_rotation_measures} and the last part of the statement are an immediate consequence of Proposition \ref{prop_kernel}, since multiples of $\Pi^*d\alpha$ induce the trivial rotation measure. 
\endproof
\endproof

\subsection{Dual rotation measures}

Let us now study more carefully the space of dual rotation measures. By Definition \ref{def_smooth_dual_area}, an element of  $\Area_{\SO(n)/\{1\},k}^{\SO(n)*}$ is represented by a form $\tau \in \Omega^{n-k,k}(\bar G)^{\bar G}$ such that $\tau \in \mathfrak{F}^{n,1},\Pi^*\alpha \wedge \tau=0,\Pi^*d\alpha \wedge \tau=0$.

As above, we have
\begin{align*}
\Omega^{n-k,k}(\bar G)^{\bar G} & \cong \Lambda^{n-k,k}(V_0^* \oplus V_\sigma^* \oplus V_\omega^* \oplus U^*) \\
& \cong \bigoplus_{\epsilon=0,1} \Lambda^\epsilon V_0^* \otimes \Lambda^{n-k-\epsilon}V_\sigma^* \otimes \bigoplus_{i=0}^k \Lambda^{k-i} V_\omega^* \otimes \Lambda^iU^*.
\end{align*}

The condition $\Pi^*\alpha \wedge \tau=0$ is satisfied for the part with $\epsilon=1$ in this sum, and the condition $\tau \in \mathfrak{F}^{n,1}$ is equivalent to $i=0$. We are thus left with the space 
\begin{displaymath}
\Lambda^{n-k-1}V_\sigma^* \otimes \Lambda^{k} V_\omega^*.
\end{displaymath}
Multiplication by the symplectic form $d\alpha$ gives a surjection  
\begin{displaymath}
L:\Lambda^{n-k-1} V_\sigma^* \otimes \Lambda^{k} V_\omega^* \to \Lambda^{n-k} V_\sigma^* \otimes \Lambda^{k+1} V_\omega^*,
\end{displaymath}
and $\Area_{\SO(n)/\{1\},k}^{\SO(n)*}$ is isomorphic to the kernel of this map. 

Let us rewrite this in more invariant terms. Let $V_\sigma^\C:=V_\sigma \otimes \C$ stand for the complexification. The group $\SL(n-1,\C)$ acts on the set of bases of $V_\sigma^\C,V_\omega^\C$ from the right in the natural way. If $Y=(Y_2,\ldots,Y_n)$ is a basis of $V_\sigma^\C$ and $g \in \SL(n-1,\C)$, then $(Yg)_i:=\sum_{j=2}^n Y_jg_{ji}$. The corresponding right operation on $V_\sigma^{*,\C},V_\omega^{*,\C}$ is given by 
\begin{align}
g^*\sigma_i & = \sum_{j=2}^n (g^{-1})_{ij}\sigma_j, \label{eq_action_sigma}\\
g^* \omega_{1i} & = \sum_{j=2}^n \omega_{1j}a_{ji}. \label{eq_action_omega}
\end{align}

As $\SL(n-1,\C)$-representations, we have $V_\sigma^\C \cong (V_\omega^{\C})^*$. The symplectic form is the canonic element of $\Lambda^2 (V^\C_\omega \oplus (V_\omega^{\C})^*)$. The above map $L$ can then be rewritten as an $\SL(n-1,\C)$-equivariant surjection 
\begin{displaymath}
L:\Lambda^{n-k-1} V^\C_\omega \otimes \Lambda^{k} (V_\omega^\C)^* \to  \Lambda^{n-k} V^\C_\omega \otimes \Lambda^{k+1} (V_\omega^{\C})^*.
\end{displaymath}
  
We may $\SL(n-1,\C)$-equivariantly identify
\begin{align*}
\Lambda^{n-k-1} V^\C_\omega \otimes \Lambda^{k} (V_\omega^{\C})^* & \cong \Lambda^k (V_\omega^{\C})^* \otimes \Lambda^k (V_\omega^{\C})^*,\\
\Lambda^{n-k} V_\omega^\C \otimes \Lambda^{k+1} (V_\omega^{\C})^* & \cong \Lambda^{k-1} (V_\omega^{\C})^* \otimes \Lambda^{k+1} (V_\omega^{\C})^*,
\end{align*}
and with this identification, the map $L:\Lambda^k (V^\C_\omega)^* \otimes \Lambda^k (V^\C_\omega)^* \to\Lambda^{k-1} (V^\C_\omega)^* \otimes \Lambda^{k+1} (V^\C_\omega)^*$ is given by
\begin{displaymath}
\tau_1 \wedge \cdots \wedge \tau_k \otimes \rho_1 \wedge \ldots \wedge \rho_k \mapsto \sum_{i=1}^k (-1)^i \tau_1 \wedge \ldots \wedge \widehat{\tau_i} \wedge \ldots \wedge \tau_k \otimes \rho_1 \wedge \ldots \wedge \rho_k \wedge \tau_i. 
\end{displaymath}

Let $\mathcal{L}_k$ be the kernel of this map. It is well-known that $\mathcal{L}_k$ is an irreducible representation of $\SL(n-1,\C)$ \cite[Exercise 15.30]{fulton_harris91}. Moreover, $\mathcal{L}_k \cdot \mathcal{L}_l \subset \mathcal{L}_{k+l}$. 

Let $\mathrm{SL}(n-1,\C)$ act on $X=(x_{ij})_{2 \leq i,j \leq n}$ (from the right) by $(g,X) \mapsto g^tXg$. Consider the map 
\begin{displaymath}
 \Psi:\C[X] \to \bigoplus_{k=0}^{n-1} \Lambda^k (V^\C_\omega)^* \otimes \Lambda^k (V^\C_\omega)^*, x_{ij} \mapsto \frac12 (\omega_{1i} \otimes  \omega_{1j}+\omega_{1j} \otimes \omega_{1i}). 
\end{displaymath}

Then $\Psi$ is an equivariant algebra morphism by \eqref{eq_action_sigma}, \eqref{eq_action_omega} and Theorem \ref{thm_ftaig}. Obviously $\Psi(x_{ij}) \in \mathcal{L}_1$. Hence, $\Psi(\C[X]_k) \subset \mathcal{L}_k$, where $\C[X]_k$ is the part of degree $k$. Since $\mathcal{L}_k$ is irreducible and the image is not trivial (consider $\Psi(x_{23} \ldots x_{2k,2k+1})$ if $k<n/2$ and $\Psi(x_{23}^2\ldots )$ otherwise), we get a surjective map $\Psi:\C[X] \to \Area_{\SO(n)/\{1\}}^{\SO(n)*} \otimes \C$. 

It is easy to check that $I \subset \ker \Psi$, hence we have a surjective map (denoted by the same letter) 
\begin{displaymath}
\Psi:\C[X]/I \to \Area_{\SO(n)/\{1\}}^{\SO(n)*} \otimes \C.
\end{displaymath}

To finish the proof, we will compare the dimensions of both sides. We first need some relations in $\C[X]/I$. 

\begin{lemma} \label{lemma_det_as_xii}
Modulo $I$, we have the following equations.
\begin{enumerate}
\item[(i)] A monomial vanishes if some index is repeated three times, i.e., $x_{aa} x_{ab}=0$, $x_{ab_1}x_{ab_2}x_{ab_3}=0$ for all $a,b,b_1,b_2,b_3 \in \{2,\ldots,n\}$.
\item[(ii)] $\det (x_{ij})_{2 \leq i,j \leq n} = \frac{n!}{2^{n-1}} x_{22}\cdot \ldots \cdot x_{n,n}$.
\item[(iii)] The elementary symmetric polynomials $E_i$ satisfy
\begin{displaymath} 
E_{i}(x_{2,2},\dots,x_{n,n})E_{j}(x_{2,2},\dots,x_{n,n}) =\binom{i+j}{i}E_{i+j}(x_{2,2},\dots,x_{n,n}).
\end{displaymath} 
\end{enumerate}
\end{lemma}

\proof
\begin{enumerate}
\item[(i)] Easy exercise.
\item[(ii)] We prove this by induction over $n$, the case $n=2$ being trivial. 

Suppose that $n>2$ and develop the determinant with respect to the last column:
\begin{displaymath}
\det (x_{ij})_{2 \leq i,j \leq n}=\sum_{l=2}^{n-1} (-1)^{l+n} x_{l,n} \det  (x_{ij})_{\substack{2 \leq i \leq n, i \neq l\\2 \leq j \leq n-1}} +x_{n,n} \det(x_{ij})_{2 \leq i,j \leq n-1}.
\end{displaymath}
Now develop the determinant in the first summand with respect to the last row. We obtain some sum of terms containing the factor $x_{l,n} x_{nj}$, which equals $-\frac12 x_{n,n}x_{lj}$. We may thus replace the factor $x_{l,n}$ by $-\frac12 x_{n,n}$, and the last row of $(x_{ij})_{\substack{2 \leq i \leq n, i \neq l\\2 \leq j \leq n-1}}$ by $x_{l2},\ldots,x_{l,n-1}$, which is just the row which was deleted. Rearranging the rows (which gives us another sign $(-1)^{n-l+1}$) we find that 
\begin{align*}
 \det (x_{ij})_{2 \leq i,j \leq n} & =\sum_{l=2}^{n-1} (-1)^{l+n} (-\frac12) (-1)^{n-l+1} x_{n,n}\det(x_{ij})_{2 \leq i,j \leq n-1} \\
 & \quad +x_{n,n} \det(x_{ij})_{2 \leq i,j \leq n-1}\\
 & = \frac{n}{2} x_{n,n} \det(x_{ij})_{2 \leq i,j \leq n-1}.
\end{align*}
\item[(iii)] By definition, 
\begin{displaymath}
E_i(x_{2,2},\dots,x_{n,n})=\sum_{|I|=i} x_{a_1,a_1}\cdot\ldots\cdot x_{a_i,a_i},
\end{displaymath}
where $I=(a_1,\dots,a_i)$ runs over all ordered subsets of size $i$ in $\{2,\dots,n\}$.

Since $x_{aa}^2=0$ for each $a$, we find that 
\begin{align*}
E_i(x_{2,2},\dots,x_{n,n}) E_j(x_{2,2},\dots,x_{n,n}) &=\sum_{|I|=i,|J|=j} \!\!\!x_{a_1,a_1}\cdot\ldots\cdot x_{a_i,a_i} x_{b_1,b_1}\cdot\ldots\cdot x_{b_j,b_j}\\
&=\binom{i+j}{i} \sum_{|K|=i+j}x_{c_1,c_1} \cdot \ldots \cdot x_{c_{i+j},c_{i+j}}\\
&=\binom{i+j}{i} E_{i+j}(x_{2,2},\dots,x_{n,n}).
\end{align*}
\end{enumerate}
\endproof

\begin{lemma} \label{lemma_upper_bound_dimension_algebra}
The dimension of the degree $k$-part of the algebra 
 \begin{displaymath}
\C[X]/I
\end{displaymath}
is at most $\frac{1}{n}\binom{n}{k}\binom{n}{k+1}$.
\end{lemma}

\begin{proof}
Let us fix $0 \leq l \leq k$ and count the number of monomials with $2l$ indices appearing once and $(k-l)$ indices appearing twice. 

There are $\binom{n-1}{k-l}$ possibilities to choose the $(k-l)$ double indices among $1,\ldots,n-1$. Replacing $x_{ij_1}x_{i j_2}$ by $-\frac12 x_{ii}x_{j_1j_2}$, we may assume that such a double index $i$ appears in a factor $x_{ii}$. 

From the remaining $(n-k+l-1)$ other indices we choose $2l$, which gives us $\binom{n-k+l-1}{2l}$ possibilities. These $2l$ indices $i_1,\ldots,i_{2l}$ will be put into pairs so that we form the product $x_{i_1i_2}\cdots x_{i_{2l-1}i_{2l}}$. 

However, we can use the relations to rule out some combinations. Arrange the $2l$ numbers $i_1,\ldots,i_{2l}$ in a circle. If a monomial contains a factor $x_{i_1 i_2} x_{i_3 i_4}$ such that the lines between $[i_1,i_2]$ and $[i_3,i_4]$ intersect, we may use the relation and replace this factor by  $x_{i_1 i_2}x_{i_3 i_4}=-x_{i_1 i_3}x_{i_4 i_2}-x_{i_1 i_4}x_{i_2 i_3}$. Note that $[i_1,i_3]$ and $[i_4,i_2]$ do not intersect and similarly $[i_1,i_4]$ and $[i_2,i_3]$ do not intersect. Continuing this way we may assume that none of the lines $[i_1,i_2],\ldots,[i_{2l-1}i_{2l}]$ do intersect. Basic combinatorics tells us that the number of such non-intersecting pairings is the Catalan number $\frac{(2l)!}{(l+1)!l!}$. 

Summarizing, we get that the dimension of the degree $k$-part of the algebra is bounded from above by 
\begin{displaymath}
 \sum_{l=0}^k\binom{n-1}{k-l}\binom{n-k+l-1}{2l}\frac{(2l)!}{l!(l+1)!}.
\end{displaymath}

It remains to see that this equals the expression given in the lemma. This can be seen by the following combinatorial argument.

Take a set of $n$ numbered cards with both sides empty. Choose $k$ among these cards and color the front side green. Independently of that, color the back side of $(k+1)$ among the $n$ cards red. The number of different colorings obtained in this way is $\binom{n}{k}\binom{n}{k+1}$. 

Start again with a set of $n$ numbered cards with both sides empty and fix a number $0 \leq l \leq k$. Choose one of the cards and declare it to be $1$-colored (the color will be fixed later). Among the remaining $(n-1)$ cards, choose $(k-l)$ and color the front side green and the back side red. Among the remaining $(n-k+l-1)$ cards, choose $2l$ and declare them to be $1$-colored. Among the $(2l+1)$ $1$-colored cards, we color $l$ in green and $(l+1)$ in red. In this way, we obtain all colorings with precisely $(k-l)$ two-colored cards, such that $k$ are green and $(k+1)$ are red. However, each of these colorings is counted $(2l+1)$ times, since any of the $(2l+1)$ $1$-colored cards can be chosen as the first card to begin with. The total number of colorings is therefore $n \sum_{l=0}^k \binom{n-1}{k-l}\binom{n-k+l-1}{2l} \binom{2l+1}{l} \frac{1}{2l+1}$. \end{proof}

\proof[End of the proof of Theorem \ref{mainthm_algebra}]
Since $\Psi:\C[X]/I \to \Area_{\SO(n)/\{1\}}^{\SO(n)*} \otimes \C$ is surjective and the dimension on the left hand side is not larger than the dimension on the right hand side, the map must be a bijection and the upper bound for the dimension from Lemma \ref{lemma_upper_bound_dimension_algebra} is attained. Obviously, the algebra isomorphism $\Psi$ induces an algebra isomorphism (denoted by the same letter)  
\begin{displaymath}
\Psi:\R[X]/I \to \Area_{\SO(n)/\{1\}}^{\SO(n)*}.
\end{displaymath}
Let us compute the image of $x_{i,j}$. Note that $x_{i,j}$ corresponds to the dual rotation measure given by the form $\tau : =*_1^{-1} \frac12(\sigma_i \wedge \omega_{j,1}+\sigma_j \wedge \omega_{i,1})$. We compute
\begin{displaymath}
 *_1^{-1}(\sigma_i \wedge \omega_{j,1})=*^{-1}\sigma_i \wedge \omega_{j,1}=(-1)^n \sigma_1 \wedge \sigma_{i_1^c} \wedge \ldots \wedge \sigma_{i_{n-2}^c} \wedge \omega_{j,1},
\end{displaymath}
and hence 
\begin{align*}
 *_1^{-1} ( \sigma_i \wedge \omega_{j,1} ) \wedge \omega_{\{i\},\{j\}} & = (-1)^n \sigma_1 \wedge \sigma_{i_1^c} \wedge \ldots \wedge \sigma_{i_{n-2}^c} \wedge \omega_{j,1} \wedge \sigma_{i} \wedge \omega_{j_1^c,1} \wedge \ldots \wedge \omega_{j_{n-2}^c,1} \wedge \rho\\
 & = - \sigma_1 \wedge \ldots \wedge \sigma_n \wedge \omega_{2,1} \wedge \ldots \wedge \omega_{n,1} \wedge \rho.
\end{align*}
It follows that 
\begin{align*}
 \langle \Psi(x_{i,i}),S_{i,i}\rangle & = - \vol \SO(n), \\
 \langle \Psi(x_{i,j}),S_{i,j}\rangle & = \langle \Psi(x_{i,j}),S_{j,i}\rangle=- \frac12 \vol \SO(n), i \neq j.
\end{align*}

In both cases, it follows that $\Psi(x_{ij})=- \vol \SO(n) \cdot \frac{S_{i,j}^*+S_{j,i}^*}{2}$.
\endproof

\begin{remark} In the following, we will identify an element of $\R[X]/I$ with its image under $\Psi$ in $\Area_{\SO(n)/\{1\}}^{\SO(n)*}$. 
\end{remark}

\begin{corollary}\label{subalgebras}
Let $H \subset \SO(n-1)$ be a closed subgroup. Then the image of the transposed globalization map $\glob^*: \Area_{\SO(n)/H}^{\SO(n)*} \to  \Area_{\SO(n)/\{1\}}^{\SO(n),*} \cong \R[X]/I$ is the subalgebra of $\R[X]/I$ consisting of all invariants under $H$ with respect to the right action $h^*X=h^tXh$. 
\end{corollary}

\begin{proof}
The statement follows directly by recalling that the inclusion $\glob^*$ preserves the algebra structure of both spaces and is $H$-equivariant.
\end{proof}

\begin{example}\label{example_Area}
We consider the case of area measures, i.e., $H=\SO(n-1)$. By Corollary~\ref{subalgebras}, $\Area^{\SO(n),*}:=\Area_{\SO(n)/\SO(n-1)}^{\SO(n),*}$ is the subalgebra of $\R[X]/I$ of all elements which are invariant under the action $h^*X=h^tXh, h\in\SO(n-1)$. Since $h^t=h^{-1}$, these invariants contain the elementary symmetric functions of $X$. By Lemma \ref{lemma_det_as_xii}(ii), they can be written as rescalings of the elementary symmetric functions of the diagonal elements $x_{ii}$, $2\leq i\leq n$. 

The linear map $\Xi:\R[t] \to (\R[X]/I)^H, t^i \mapsto i!E_i(x_{2,2},\ldots,x_{n,n})$ induces by Lemma \ref{lemma_det_as_xii} an injective algebra morphism 
\begin{displaymath}
\tilde \Xi:\R[t]/\langle t^{n}\rangle \to  (\R[X]/I)^H = \Area^{\SO(n),*}. 
\end{displaymath}
Since the dimensions on both sides agree, this map is an algebra isomorphism. Hence, 
\begin{displaymath}
\Area^{\SO(n),*} \cong\R[t]/\langle t^{n}\rangle,
\end{displaymath}
which is of course well-known. 

The additive kinematic formulas are given by 
\begin{equation} \label{eq_kinematic_area}
 A(S_i)=\frac{1}{\omega_n} \sum_{k+l=i} \binom{i}{k} S_k \otimes S_l,
\end{equation}
see \cite{bernig_hug}, or Schneider \cite[Theorem 4.4.6]{schneider_book14}. It follows that 
\begin{displaymath}
 S_k^* \cdot S_l^*=\frac{1}{\omega_n} \binom{k+l}{k} S_{k+l}^*.
\end{displaymath}
Clearly $t$ is mapped to some multiple $c S_1^*$. We will see later that $c=\omega_n$. Then $t^k$ is mapped to $\omega_n k! S_k^*$, as can be shown by induction over $k$.
\end{example}


\section{Algebraic structure of $\FlagArea^{(p),\SO(n),*}$} \label{sec_algebraic_structure_flag_case}

The aim in this section is to prove Proposition \ref{prop_algebra_flag_case}. Recall first that $\sigma_i,\omega_{ij}$ are the coordinates of the Maurer-Cartan form of $\overline{\SO(n)}$. The volume form of the unit sphere is $\omega_{21} \wedge \ldots \wedge \omega_{n1}$. We have the structure equations 
\begin{displaymath}
d\sigma_i=- \sum_{j=1}^n \omega_{ij} \wedge \sigma_j, d\omega_{ij}=-\sum_{k=1}^n \omega_{ik} \wedge \omega_{kj}.
\end{displaymath}
Unwinding the definitions in Section \ref{sec_rotation_measures}, we have 
\begin{displaymath}
x_{ij}=\frac12 (\sigma_i \wedge \omega_{j1}+\sigma_j \wedge \omega_{i1});
\end{displaymath}
in particular $x_{ii}=\sigma_i \wedge \omega_{i1}$. 

\begin{proof}[Proof of Proposition~\ref{prop_algebra_flag_case}]
Recall that $\Flag_{1,p+1} \cong \SO(n)/H$ with $H=\mathrm{S}(\O(p)\times\O(q))$. By Corollary~\ref{subalgebras}, $\glob^*$ maps $\FlagArea^{(p),\SO(n),*}$ bijectively to the algebra of $H$-invariant elements in $\R[X]/I$, where the action is given by $h^*X:=h^tXh=h^{-1}Xh$. 

Assume first that $p \neq q$. We claim that the linear map  
\begin{displaymath}
\Xi:\R[x,y] \to (\R[X]/I)^H,
\end{displaymath}
which sends $x^iy^j$ to $i! E_i(x_{2,2},\ldots,x_{p+1,p+1}) j! E_j(x_{p+2,p+2},\ldots,x_{n,n})$ is an algebra morphism.

It is clear that the image of each monomial is $H$-invariant. The compatibility with the product follows from Lemma \ref{lemma_det_as_xii}.

Obviously, $x^{p+1},y^{q+1} \in \ker \Xi$, hence there is an induced algebra morphism
\begin{equation}\label{tilde_Xi}
\tilde \Xi: \R[x,y]/\langle x^{p+1},y^{q+1}\rangle \to (\R[X]/I)^H.
\end{equation}

We claim that this map is injective. To do so, introduce a bigrading on $(\R[X]/I)^H$ by declaring that 
\begin{displaymath}
\deg x_{ab}= \begin{cases} (1,0) & 2 \leq a,b \leq p+1\\ 
(\frac12,\frac12) & 2 \leq a \leq p+1, p+2 \leq b \leq n \\
 (\frac12,\frac12) & 2 \leq b \leq p+1, p+2 \leq a \leq n\\
(0,1) & p+2 \leq a,b \leq n.
\end{cases} 
\end{displaymath}
It is easily checked that the ideal $I$ is bigraded, so we indeed have a bigrading on the quotient. The image of $x^iy^j, 0 \leq i \leq p, 0 \leq j \leq q$ is of bidegree $(i,j)$. To prove injectivity of $\tilde \Xi$, it is therefore enough to prove that $\tilde \Xi(x^iy^j) \neq 0$ for $0 \leq i \leq p, 0 \leq j \leq q$. But $\tilde \Xi(x^iy^j)$ corresponds to the form 
\begin{equation} \label{eq_xiyj}
i!  E_i(\sigma_2 \wedge \omega_2,\ldots,\sigma_{p+1} \wedge \omega_{p+1}) j!  E_j(\sigma_{p+2} \wedge \omega_{p+2},\ldots,\sigma_n \wedge \omega_n),
\end{equation}
which is obviously non zero. 

To conclude the proof that $\tilde \Xi$ is an algebra isomorphism, it is enough to compare dimensions. The dimension of the $k$-homogeneous part of the left hand side is the number of monomials $x^iy^j$ with $i+j=k, 0 \leq i \leq p, 0 \leq j \leq q$, which is easily computed as $\min\{p,q,k,n-k-1\}+1$. The $k$-homogeneous part on the right hand side is isomorphic to 
$\FlagArea^{(p),\SO(n)}_{k}$, which is of the same dimension by \cite[Theorem 4]{abardia_bernig_dann}. 

Let us now consider the case $p=q$. The above proof goes through word by word, except that $\dim \FlagArea^{(p),\SO(n)}_{k}=\min\{p,q,k,n-k-1\}+2$ if $k=p=q$. We define 
\begin{displaymath}
\Xi:\R[x,y,u] \to (\R[X]/I)^H, 
\end{displaymath} 
similarly as above, with 
\begin{displaymath}
\Xi(u):=\det (x_{ij})_{\substack{2 \leq i \leq p+1\\p+2 \leq j \leq n}}.
\end{displaymath}

By Lemma \ref{lemma_det_as_xii}, we have $\Xi(xu)=\Xi(yu)=0$. More precisely, each monomial in the development of the determinant contains each index at least once, while each term in $x$ (resp. $y$) contains each index twice. 

We next compute $\Xi(u^2)$ by using $x_{ai}x_{aj}=-\frac12 x_{aa}x_{ij}$, which is a consequence of Lemma \ref{lemma_det_as_xii}. 
\begin{align*}
\Xi(u^2)&=\left(\sum_{\sigma \in \Sigma_{p}} \sgn(\sigma) x_{2,\sigma(p+2)}\ldots x_{p+1,\sigma(n)}\right)\!\!\! \left(\sum_{\pi \in\Sigma_{p}} \sgn(\pi) x_{2,\pi(p+2)}\ldots x_{p+1,\pi(n)}\right)\\
&=\left(\frac{-1}{2}\right)^p x_{2,2}\cdot\ldots\cdot x_{p+1,p+1}\! \sum_{\sigma,\pi \in\Sigma_p}\! \sgn(\sigma)\sgn(\pi) x_{\sigma(p+2)\pi(p+2)} \ldots x_{\sigma(n)\pi(n)}\\
&=\left(\frac{-1}{2}\right)^p x_{2,2}\cdot\ldots\cdot x_{p+1,p+1} p!\sum_{\pi\in\Sigma_p}\sgn(\pi)x_{p+2,\pi(p+2)}\ldots x_{n,\pi(n)}\\
&=\left(-\frac12\right)^p p! E_{p}(x_{2,2},\ldots, x_{p+1,p+1})\det(x_{i,j})_{p+2\leq i,j\leq n}\\
&= \left(-\frac12\right)^p p! E_{p}(x_{2,2},\ldots, x_{p+1,p+1}) \frac{(p+1)!}{2^p} E_p(x_{p+2,p+2},\ldots,x_{n,n})\\
& = (-1)^p \frac{(p+1)}{ 2^{2p} } \Xi(x^py^p).
\end{align*}

It follows that there is an induced algebra morphism
\begin{displaymath}
\tilde \Xi:\R[x,y,u]/\langle x^{p+1},y^{p+1},xu,yu,u^2-(-1)^p \frac{ (p+1) }{  2^{2p} }x^py^p\rangle \to (\R[X]/I)^H.
\end{displaymath}  
We argue as above to prove that $\tilde \Xi$ is injective. However, if $p$ is even, there are two elements whose images under $\tilde \Xi$ are of bidegree $(\frac{p}{2},\frac{p}{2})$, namely $x^\frac{p}{2}y^\frac{p}{2}$ and $u$. We show that the images of these elements are linearly independent. 

Let $\{X_i\}_{2\leq i\leq p+1}$ be the dual basis to $\{\sigma_{i}\}_{2\leq i\leq p+1}$ and let $\{X_{i,1}\}_{p+2\leq i\leq n}$ be the dual basis to $\{\omega_{i,1}\}_{p+2\leq i\leq n}$. 
By \eqref{eq_xiyj}, the form corresponding to $\tilde\Xi(x^{\frac{p}{2}}y^{\frac{p}{2}})$ is non-zero but vanishes evaluated at $(X_2,\dots,X_{p+1},X_{p+2,1},\dots,X_{n,1})$. For the associated form to $\tilde \Xi(u)$, we have 
\begin{align*}
&\det\left(\frac12 (\sigma_{i}\wedge\omega_{j,1} + \sigma_j\wedge\omega_{i,1})\right)_{\substack{2 \leq i \leq p+1\\p+2 \leq j \leq n}}(X_2,\dots,X_{p+1},X_{p+2,1},\dots,X_{n,1})\\
&= \frac{1}{2^p}\! \sum_{\pi \in \Sigma_p} \sgn(\pi) \sigma_2 \wedge \omega_{\pi(p+2),1} \wedge \dots \wedge\sigma_{p+1} \wedge \omega_{\pi(n),1}(X_2,\dots,X_{p+1},X_{p+2,1},\dots,X_{n,1})\\
&=(-1)^{\frac{p(p-1)}{2}}\frac{p!}{2^p}\neq 0.
\end{align*}

Finally, comparing dimensions again, we see that $\tilde \Xi$ is an algebra isomorphism.
\end{proof}

\section{Bases for $\FlagArea^{(p),\SO(n)}$ and $\FlagArea^{(p),\SO(n),*}$} \label{sec_bases_flag_case}

 In the following, $0\leq p\leq n-1, q:=n-p-1, m_k:=\min\{p,q,k,n-k-1\}$.

\subsection{Bases for $\FlagArea^{(p),\SO(n)}$}

We first recall the bases of $\FlagArea^{(p),\SO(n)}$ introduced in \cite{abardia_bernig_dann}.

For $\max\{0,k-q\} \leq a \leq \min\{k,p\}$, $\hat \eta_{k,a} \in \Omega^{n-1}(V \times \Flag_{1,p+1})$ is defined as the coefficient of $\alpha^a \beta^{k-a}$ in the expansion of 
\begin{equation}\label{expansion_hat_tau}
 \hat \eta_{\alpha,\beta}:=\bigwedge_{i=2}^{p+1}(\alpha \sigma_i+\omega_{i,1}) \wedge \bigwedge_{j=p+2}^n (\beta \sigma_{j}+\omega_{j,1}).
\end{equation}

In the case $2p=2k=n-1$ we set
\begin{equation}\label{omega_especial}
\hat \eta_{ex} := \sigma_{p+2} \wedge \dots \wedge \sigma_n \wedge \omega_{p+2,1} \wedge \dots \wedge \omega_{n,1} \in \Omega^{n-1}(V \times \Flag_{1,p+1}).
\end{equation}

We denote
\begin{equation}\label{eq_omega_a}
\hat \omega_{k,a} := \frac{\omega_n}{\vol(\Flag_{1,p+1})} \hat \eta_{k,a} \wedge \rho \in \Omega^m(V \times \Flag_{1,p+1}),\ \max\{0,k-q\}\leq a\leq\min\{k,p\}, 
\end{equation}
and, if $n$ is odd and $2p=n-1$, 
\begin{equation}\label{eq_omega_ex}
\hat \omega_{ex}:=\frac{\omega_n}{\vol(\Flag_{1,p+1})}\hat \eta_{ex} \wedge \rho \in \Omega^m(V \times \Flag_{1,p+1}).
\end{equation}
Here $\rho$ denotes the volume form of the fiber of the map $\Pi:V\times\Flag_{1,p+1}\to V\times S^{n-1}$. We remark that the factor in the definition of $\hat \omega_{k,a}$ does not appear explicitly in \cite{abardia_bernig_dann}, but implicitly by the fact that the volume form on the fiber should be normalized to volume $1$ (see \cite[Corollary 4.6]{abardia_bernig_dann}).

The smooth flag area measure associated to the form $\hat\omega_{k,a}$ (resp. $\hat \omega_{ex}$) is denoted by $\Phi_{k,a}\in\FlagArea^{(p),\SO(n)}$ (resp. $\Phi_{ex}$) and is defined for $0\leq p,k\leq n-1$, $\max\{0,k-q\}\leq a\leq\min\{p,k\}$. 

Another basis for $\FlagArea^{(p),\SO(n)}$ was given in \cite{abardia_bernig_dann}. This basis contains the elements $S_k^{(p)}$ in $\FlagArea^{(p),\SO(n)}$ previously introduced in \cite{hinderer_hug_weil}. The smooth flag area measures of this basis are denoted by $S_{k}^{(p),i}$ and defined, for $0\leq p,k\leq n-1$, $0\leq i\leq m_{k}$, as
\begin{equation}\label{eq_S_in_Phi}
S_k^{(p),i}= c_{n,k,p,i} \sum_{a=\min\{p,k\}-m_{k}}^{\min\{p,k\}-i} \binom{\min\{p,k\}-a}{i} \Phi_{k,a},
\end{equation}
where 
\begin{displaymath}
c_{n,k,p,i}=\binom{n-1}{k}^{-1}\binom{m_{k}}{i}^{-1}\binom{|k-q|+m_{k}}{i}^{-1}\binom{n-1}{i}.
\end{displaymath}
For the exceptional case $2p=2k=n-1$, the following notation is used:
$$\tilde S_{\frac{n-1}{2}}^{(\frac{n-1}{2})}=\Phi_{ex}.$$

For the indicated ranges, both sets $\{S_k^{(p),i}\}$ and $\{\Phi_{k,a}\}$ constitute a basis of $\FlagArea^{(p),\O(n)}$ and the sets $\{S_k^{(p),i},\tilde S_{\frac{n-1}{2}}^{(\frac{n-1}{2})}\}$ and $\{\Phi_{k,a},\Phi_{ex}\}$ a basis of $\FlagArea^{(p),\SO(n)}$. A straightforward computation shows that the inverse relation is given as follows.

\begin{lemma}\label{lemma:lincom}
\begin{displaymath}
\Phi_{j,a}=\sum_{s=\min\{p,j\}-a}^{\min\{p,q,j,n-j-1\}} (-1)^{a+\min\{p,j\}+s}c_{n,j,p,s}^{-1}\binom{s}{\min\{p,j\}-a}S_j^{(p),s}.
\end{displaymath}
\end{lemma}

\begin{lemma} \label{lemma_glob_phi}
Let $\glob:\FlagArea^{(p),\SO(n)} \to \Area^{\SO(n)}$ be the globalization map. Then 
\begin{align*}
\glob S_k^{(p),i} & = S_k,\\
\glob \Phi_{k,a} & =\binom{q}{k-a}\binom{p}{a} S_k.
\end{align*}
\end{lemma}

\proof
The first equation was shown in \cite[Theorem 3]{abardia_bernig_dann}. The second equation can be deduced from the first one and Lemma \ref{lemma:lincom}.
\endproof

\subsection{Bases for $\FlagArea^{(p),\SO(n),*}$}

\begin{corollary} \label{cor_dual_bases}
 The dual bases are related by 
$$
  \Phi_{k,a}^* = \sum_{i=0}^{\min\{k,n-k-1,p,q\}} c_{n,k,p,i} \binom{\min\{k,p\}-a}{i} S_k^{(p),i,*},$$
  $$
  S_k^{(p),i,*} = c_{n,k,p,i}^{-1} \sum_{a=\min\{p,k\}-i}^{\min\{p,k\}} (-1)^{a+\min\{p,k\}+i}\binom{i}{\min\{p,k\}-a} \Phi_{k,a}^*.
$$
\end{corollary}

\medskip
Let us introduce another basis of $\FlagArea^{(p),\SO(n),*}$. 
For $\max\{0,k-q\} \leq a \leq \min\{k,p\}$, we define $ \tilde\tau_{k,a} \in \mathcal{J}^{n,tr}$ by
\begin{equation}\label{expansion_tau_tilde}
\tilde\tau_{k,a}:=*_1^{-1} E_{a}(x_{22},\ldots,x_{p+1,p+1}) E_{k-a}(x_{p+2,p+2},\ldots,x_{n,n}).
\end{equation}

Let $\tilde\Phi_{k,a}^*$ be the element in $\FlagArea^{(p),\SO(n),*}$ given by $\tilde \tau_{k,a}$. If $(k,p) \neq (\frac{n-1}{2},\frac{n-1}{2})$, the elements of $\{\tilde\Phi_{k,a}^*\}_{\max\{0,k-q\}\leq a\leq\min\{p,k\}}$ constitute a basis of $\FlagArea^{(p),\SO(n),*}_k$ by \cite[Theorem~4]{abardia_bernig_dann}.

In the exceptional case $2p=2k=n-1$, we define
\begin{displaymath}
\tilde \tau_{ex}:=*_1^{-1} \det(x_{ij})_{\substack{2 \leq i \leq p+1\\p+2 \leq j \leq n}} \in\mathcal{J}^{n,tr}.
\end{displaymath}

The smooth dual flag area measure associated to the form $\tilde\tau_{ex}$ will be denoted by $\tilde\Phi_{ex}^*$. For $2p=2k=n-1$, the elements of $\{\tilde\Phi_{k,a}^*\}_{0\leq a\leq k}$ together with $\tilde\Phi_{ex}^*$ constitute a basis of $\FlagArea^{(p),\SO(n),*}_k$. This follows again from \cite[Theorem~4]{abardia_bernig_dann} and the fact that $\hat\tau_{ex} \wedge \tilde\tau_{ex} \neq 0$ but $\hat\tau_{k,a} \wedge \tilde\tau_{ex}=0$. 

The following lemma expresses the elements of the dual basis of $\{\Phi_{k,a}\}\cup\{\Phi_{ex}\}$ and $\{S_k^{(p),i}\}$ in the algebra given in Proposition~\ref{prop_algebra_flag_case}. 

\begin{lemma}\label{lemma_dual_omega}
\begin{enumerate}
\item[(i)] For $0\leq k\leq n-1$ and $\max\{0,k-q\}\leq a\leq\min\{p,k\}$
\begin{displaymath}
\Phi_{k,a}^*
=\frac{1}{\omega_n} \binom{q}{k-a}^{-1}\binom{p}{a}^{-1}\frac{1}{a!(k-a)!}x^ay^{k-a}. 
\end{displaymath}
\medskip
\item[(ii)] For $0\leq k\leq n-1$ and $0\leq i\leq \min\{p,q,k,n-k-1\}$, 
\begin{align*}
S_{k}^{(p),i,*} & =\frac{(-1)^{i+\min\{p,k\}}}{\omega_n c_{n,k,p,i}}  \cdot \\
& \quad \cdot \sum_{a=\min\{p,k\}-i}^{\min\{p,k\}} (-1)^a \frac{1}{a!(k-a)!} \binom{i}{\min\{p,k\}-a}\binom{q}{k-a}^{-1}\binom{p}{a}^{-1}x^ay^{k-a}.
\end{align*}

\medskip
\item[(iii)] If $2p=2k=n-1$, then
\begin{align*}
\Phi_{ex}^* &=\frac{(-1)^p}{\omega_n\, p!}u.
\end{align*}
\end{enumerate}
\end{lemma}

\begin{proof}
\begin{enumerate}
\item[(i)]
Let
\begin{align}
\tilde \tau_{\alpha,\beta} & :=*_1^{-1} \sum_{i=0}^p \alpha^{p-i} E_i(\sigma_2 \wedge \omega_{2,1},\ldots,\sigma_{p+1} \wedge \omega_{p+1,1}) \sum_{j=0}^q \beta^{q-j} E_j(\sigma_{p+2} \wedge \omega_{p+2,1},\ldots,\sigma_{n} \wedge \omega_{n,1}) \nonumber \\
& = \sigma_1\wedge\bigwedge_{i=2}^{p+1}(\alpha \sigma_i+\omega_{i,1})\wedge \bigwedge_{j=p+2}^n (\beta \sigma_j+\omega_{j,1}), \label{eq_expansion_tilde_tau}
\end{align}
where the second equation follows from  \cite[Eq. (47)]{bernig_fu06}. Then the coefficient of $\alpha^{p-a} \beta^{q-k+a}$ in $\tilde \tau_{\alpha,\beta}$ is $\tilde \tau_{k,a}$. 

Let $k,a$ be fixed. We first compute $\langle \tilde\Phi_{l,b}^*,\Phi_{k,a}\rangle$ for fixed indices $l,b$ by using the expansions in \eqref{expansion_hat_tau} and \eqref{eq_expansion_tilde_tau}. For $\alpha,\beta,\tilde\alpha,\tilde\beta\in\R$, one can easily check that 
\begin{displaymath}
\hat\tau_{\alpha,\beta} \wedge \tilde\tau_{\tilde\alpha,\tilde\beta} =(\alpha+\tilde\alpha)^p (\beta+\tilde\beta)^{n-p-1}\sigma_1 \wedge \sigma_2 \wedge \dots \wedge \sigma_n \wedge\omega_{2,1} \wedge \dots \wedge \omega_{n,1}.
\end{displaymath}
Since $\hat\tau_{k,a}\wedge\tilde\tau_{l,b}$ is the coefficient of $\alpha^{a}\beta^{k-a}\tilde\alpha^{p-b}\tilde\beta^{n-l-1-p+b}$ in this expression, we find that $\hat\tau_{k,a}\wedge\tilde\tau_{l,b}=0$ if $b\neq a$ or $l\neq k$. In the case $b=a$ and $l=k$, we have, by Definition~\ref{def_smooth_dual_area},

\begin{align*}
\langle \tilde\Phi_{k,a}^*,\Phi_{k,a}\rangle &= \frac{\omega_n}{\vol(\Flag_{1,p+1})} \int \tilde\tau_{k,a} \wedge \hat\tau_{k,a} \wedge \rho\\
& = 
 \frac{\omega_n}{\vol(\Flag_{1,p+1})} \binom{p}{a}\binom{n-p-1}{k-a}\int \sigma_1\wedge\dots\wedge\sigma_n\wedge\omega_{2,1}\wedge\dots\wedge\omega_{n,1}\wedge\rho
\\&=\omega_n \binom{p}{a}\binom{n-p-1}{k-a}.  
\end{align*}
Hence, 
\begin{equation}\label{dual_in_tau}
\Phi^*_{k,a}= \frac{1}{\omega_n} \binom{q}{k-a}^{-1}\binom{p}{a}^{-1}\tilde\Phi_{k,a}^*.
\end{equation}

The statement now follows from the fact that $\tilde \Phi_{k,a}^*= \frac{x^{a}}{a!} \frac{y^{k-a}}{(k-a)!}$, which is a consequence of Lemma \ref{lemma_det_as_xii}(iii).

\item[(ii)] This follows from Corollary \ref{cor_dual_bases} and (i). 

\item[(iii)] We first show that 
\begin{displaymath}
\Phi_{ex}^*=\frac{(-1)^{p}}{\omega_n\, p!}\tilde\Phi_{ex}^*,
\end{displaymath}
which follows from 
\begin{align*}
&\langle \tilde\Phi_{ex}^*,\Phi_{ex}\rangle=\frac{\omega_n}{\vol(\Flag_{1,p+1})}
\int \tilde \tau_{ex} \wedge \sigma_{p+2}\wedge\dots\wedge\sigma_n\wedge\omega_{p+2,1}\wedge\dots\wedge\omega_{n,1} \\
&=\frac{\omega_n}{\vol(\Flag_{1,p+1})}
(-1)^{\binom{p}{2}} p!
\int *_1^{-1}(\sigma_{p+2}\wedge \dots \wedge \sigma_n \wedge \omega_{2,1} \wedge \dots \wedge \omega_{p+1,1})\\
& \quad  \wedge \sigma_{p+2}\wedge\dots\wedge\sigma_n\wedge\omega_{p+2,1}\wedge\dots\wedge\omega_{n,1} \wedge \rho\\
&=\frac{\omega_n}{\vol(\Flag_{1,p+1})}\, p!
\int \sigma_{1}\wedge \dots \wedge \sigma_{p+1} \wedge \omega_{2,1} \wedge \dots \wedge \omega_{p+1,1} \\
& \quad \wedge \sigma_{p+2}\wedge\dots\wedge\sigma_n\wedge\omega_{p+2,1}\wedge\dots\wedge\omega_{n,1} \wedge \rho\\
&=\frac{\omega_n}{\vol(\Flag_{1,p+1})}
(-1)^{p} p!
\int \sigma_{1}\wedge \dots \wedge \sigma_n \wedge \omega_{2,1} \wedge \dots \wedge \omega_{n,1} \wedge \rho\\
&=(-1)^{p}p!\omega_n.
\end{align*}
The proof is finished by noting that $\tilde \Phi_{ex}^*=u$.
\end{enumerate}
\end{proof}

The globalization map $\glob:\FlagArea^{(p),\SO(n)} \to \Area^{\SO(n)}$ induces an algebra morphism $\glob^*:\Area^{\SO(n),*} \to \FlagArea^{(p),\SO(n),*}$. In terms of the algebraic descriptions it is given as $t \mapsto x+y$. This follows from Lemma \ref{lemma_glob_dual} since $t$ is mapped to $E_1(x_{22},\ldots,x_{nn})=E_1(x_{22},\ldots,x_{p+1,p+1})+E_1(x_{p+2,p+2},\ldots,x_{nn})$. By Lemma \ref{lemma_dual_omega} and \eqref{eq_S_in_Phi} we thus have
\begin{align*}
\langle t,S_1\rangle & = \langle t,\glob S_1^{(p),0}\rangle = \langle \glob^*(t),S_1^{(p),0}\rangle\\
& =\left\langle x+y,S_1^{(p),0}\right\rangle =\left\langle \omega_n(p\Phi_{1,1}^*+q\Phi_{1,0}^*),\frac{1}{n-1} (\Phi_{1,0}+\Phi_{1,1})\right\rangle=\omega_n,
\end{align*}
which implies that $t=\omega_n S_1^*$. Hence the constant $c$ from the end of Section~\ref{sec_rotation_measures} equals $\omega_n$.

\section{Explicit additive kinematic formulas} \label{sec_explicit_kinematic_flag_case}

The aim in this section is to obtain explicit additive kinematic formulas for $\FlagArea^{(p),\SO(n)}$. We denote by
$$A_{1,p+1}^G:\FlagArea^{(p),G}\to\FlagArea^{(p),G}\otimes\FlagArea^{(p),G}$$
the additive kinematic operator for flag area measures in $\FlagArea^{(p),G}$ with $G$ either $\O(n)$ or $\SO(n)$.

\begin{theorem} \label{thm_kinformulas_phi}
\begin{enumerate}
\item[(i)] For $0\leq k\leq n-1$ and $\max\{0,k-q\}\leq a\leq\min\{p,k\}$, the additive kinematic formulas for $\O(n)$ are given by
\begin{displaymath}
A_{1,p+1}^{\O(n)}(\Phi_{k,a})=\frac{1}{\omega_n} \sum_{j=0}^k\sum_{b=\max\{0,j-k+a\}}^{\min\{a,j\}}
c^{k,a}_{j,b}\Phi_{j,b}\otimes\Phi_{k-j,a-b},
\end{displaymath}
where
\begin{displaymath}
c^{k,a}_{j,b}=\binom{q}{j-b}^{-1}\binom{p}{b}^{-1}\binom{q-k+j+a-b}{j-b}\binom{p-a+b}{b}.
\end{displaymath}
\item[(ii)] The additive kinematic formulas for $\SO(n)$ are given as follows. If $(k,p) \neq (n-1,\frac{n-1}{2})$, then
\begin{displaymath}
A^{\SO(n)}_{1,p+1}(\Phi_{k,a})=A_{1,p+1}^{\O(n)}(\Phi_{k,a}).
\end{displaymath}
If $k=2p=n-1$, then
\begin{displaymath}
A^{\SO(n)}_{1,p+1}(\Phi_{n-1,\frac{n-1}{2}})=A_{1,p+1}^{\O(n)}(\Phi_{n-1,\frac{n-1}{2}})+\frac{(-1)^p(p+1)}{2^{2p}  \omega_n}\Phi_{ex}\otimes\Phi_{ex},
\end{displaymath}
and
\begin{displaymath}
A^{\SO(n)}_{1,p+1}(\Phi_{ex})=\frac{1}{\omega_n}\left(\Phi_{ex}\otimes\Phi_{0,0}+\Phi_{0,0}\otimes\Phi_{ex}\right).
\end{displaymath}
\end{enumerate}
\end{theorem}

\begin{proof}
\begin{enumerate}
\item[(i)] Using the algebraic structure of $\FlagArea^{(p),\O(n),*}$ and Theorem~\ref{thm_ftaig} we compute the coefficient $c^{k,a}_{j,b,l,c}$ of $\Phi_{j,b} \otimes \Phi_{l,c}$ in $A^{\O(n)}_{1,p+1}(\Phi_{k,a})$ as
\begin{align*}
c^{k,a}_{j,b,l,c} & =\left\langle A^{\O(n)}_{1,p+1}(\Phi_{k,a}),\Phi_{j,b}^*\otimes\Phi_{l,c}^*\right\rangle\\
& =\left\langle \Phi_{k,a},A_{1,p+1}^{\O(n),*}(\Phi_{j,b}^*\otimes\Phi_{l,c}^*)\right\rangle\\
& =\langle \Phi_{k,a},\Phi_{j,b}^* \cdot \Phi_{l,c}^* \rangle.
\end{align*}

We have 
\begin{align*}
\Phi_{j,b}^* \cdot \Phi_{l,c}^* & = \frac{1}{\omega_n^2} \binom{q}{j-b}^{-1} \binom{p}{b}^{-1}\binom{q}{l-c}^{-1}\binom{p}{c}^{-1}
\\&\qquad \cdot\frac{1}{b!(j-b)!c!(l-c)!}x^{b}y^{j-b}x^{c}y^{l-c}\\
& =  \frac{1}{\omega_n}\binom{q}{j-b}^{-1} \binom{p}{b}^{-1}\binom{q}{l-c}^{-1}\binom{p}{c}^{-1}
\\&\qquad \cdot \binom{b+c}{b} \binom{j+l-b-c}{j-b} \binom{q}{j+l-b-c}\binom{p}{b+c}  \Phi_{j+l,b+c}^*.
\end{align*}

It follows that $c^{k,a}_{j,b,l,c}$ equals zero unless $k=j+l$ and $a=b+c$ and the result follows. 

Let us double check the constants in this formula. Clearly the additive kinematic formulas commute with the globalization map $\glob:\FlagArea^{(p),\O(n)} \to \Area^{\O(n)}$. By Lemma \ref{lemma_glob_phi}, $\glob \otimes \glob$ applied to the right hand side gives 
\begin{align*}
\frac{1}{\omega_n} & \sum_{j=0}^k\sum_{b=\max\{0,j-k+a\}}^{\min\{a,j\}} 
c^{k,a}_{j,b} \glob \Phi_{j,b} \otimes \glob \Phi_{k-j,a-b} \\
& = \frac{1}{\omega_n} \sum_{j=0}^k\sum_{b=\max\{0,j-k+a\}}^{\min\{a,j\}}
c^{k,a}_{j,b} \binom{q}{j-b} \binom{p}{b} S_{j} \otimes \binom{q}{k-j-a+b} \binom{p}{a-b} S_{k-j} \\
& = \frac{1}{\omega_n} \sum_{j=0}^k \sum_{b=\max\{0,j-k+a\}}^{\min\{a,j\}}
\binom{q-k+j+a-b}{j-b}\binom{p-a+b}{b} \binom{q}{k-j-a+b} \binom{p}{a-b} S_{j} \otimes S_{k-j} \\
& = \frac{1}{\omega_n} \binom{q}{k-a} \binom{p}{a} \sum_{j=0}^k  \sum_{b=\max\{0,j-k+a\}}^{\min\{a,j\}}
\binom{a}{b}   \binom{k-a}{j-b} S_{j} \otimes S_{k-j} \\
& = \frac{1}{\omega_n} \binom{q}{k-a} \binom{p}{a} \sum_{j=0}^k  \binom{k}{j} S_{j} \otimes S_{k-j}.
\end{align*}
The globalization of the left hand side in the kinematic formula is $\binom{q}{k-a} \binom{p}{a}A(S_k)$, which equals the globalization of the right hand side by \eqref{eq_kinematic_area}.

\item[(ii)] 
For the proof of additive kinematic formulas for $\SO(n)$, we first observe that the above argument remains the same unless we are in the exceptional case $2p=k=n-1$. In this case, it remains to compute the coefficient $\langle A_{1,p+1}^{\SO(n)}(\Phi_{n-1,\frac{n-1}{2}}),\Phi_{ex}^*\otimes\Phi_{ex}^*\rangle$. For that, we compute
\begin{align*}
\nonumber (\Phi_{ex}^*)^2&=\frac{1}{\omega_n^2 p!^2} u^2
\\&=\frac{(-1)^p(p+1)}{\omega_n^2 2^{2p} p!^2} x^{p}y^{p}
\\&=\frac{(-1)^p (p+1)}{2^{2p} \omega_n} \Phi_{n-1,\frac{n-1}{2}}^*
\end{align*}
from which the result follows. 

Finally, the additive kinematic formula for $\Phi_{ex}$ has to be a linear combination of $\Phi_{ex}\otimes\Phi_{k,0}$, $\Phi_{k,0}\otimes\Phi_{ex}$ and the formula follows directly.
\end{enumerate}
\end{proof}

\begin{corollary}
Let $0\leq p\leq n-k-1$ and consider the flag area measures $S_{k}^{(p)}$ introduced in \cite{hinderer_hug_weil}. Then,
\begin{displaymath}
A^{\SO(n)}_{1,p+1}(S_k^{(p)})=\frac{1}{\omega_{n-p}}\sum_{j=0}^k \binom{k}{j} S_j^{(p)}\otimes S_{k-j}^{(p)}.
\end{displaymath}
\end{corollary}

\begin{proof}
This follows from the case $a=0$ in Theorem \ref{thm_kinformulas_phi}, since 
\begin{displaymath}
S_k^{(p)}=\frac{\omega_{n-p}}{\omega_{n}}S_k^{(p),\min\{k,p\}}=\frac{\omega_{n-p}}{\omega_{n}} \binom{q}{k}^{-1} \Phi_{k,0}\end{displaymath}
by \cite[Theorem~3]{abardia_bernig_dann} and by \eqref{eq_S_in_Phi}. 
\end{proof}

\proof[Proof of Theorem \ref{thm_kin_S}]
By the previous results, it suffices to compute $A_{1,p+1}^{\O(n)}$. Indeed, for $2p=n-1$, by definition $\tilde S^{(\frac{n-1}{2})}_{\frac{n-1}{2}}=\Phi_{ex}$ and $\Phi_{n-1,\frac{n-1}{2}}=S_{n-1}^{(p),0}$, hence $A_{1,p+1}^{\SO(n)}$ can be directly deduced from $A_{1,p+1}^{\O(n)}$ and Theorem~\ref{thm_kinformulas_phi}.

By proceeding as in Theorem~\ref{thm_kinformulas_phi}, we have that the coefficient of $S_{j}^{(p),b}\otimes S_{l}^{(p),c}$ in $A^{\O(n)}_{1,p+1}(S_k^{(p),i})$ is given by
\begin{displaymath}
\langle S_{k}^{(p),i}, S_{j}^{(p),b,*} \cdot S_{l}^{(p),c,*}\rangle.
\end{displaymath}

Let 
\begin{displaymath}
S_{k}^{(p),i,*}=\sum_{a=m'_{k}-i}^{m_{k}}K_{k,i,a}x^ay^{k-a}
\end{displaymath}
with $K_{k,i,a}$ the constant given in Lemma~\ref{lemma_dual_omega}(ii) and let
\begin{displaymath}
x^ay^{k-a}=\tilde K_{k,a} \Phi_{k,a}^*
\end{displaymath}
with $\tilde K_{k,a}$ the constant obtained from Lemma~\ref{lemma_dual_omega}(i). 

Using the defined constants and \eqref{eq_S_in_Phi}, we have
\begin{align*}
\langle S&_{k}^{(p),i}, S_{j}^{(p),b,*} \cdot S_{l}^{(p),c,*}\rangle =
\left\langle S_{k}^{(p),i},\sum_{s=m'_{j}-b}^{m_{j}}\sum_{r=m'_{l}-c}^{m_{l}}K_{j,b,s}K_{l,c,r}\tilde K_{j+l,s+r}\Phi_{j+l,r+s}^{*}\right\rangle\\
& =\sum_{t=m'_{k}-m_{k}}^{m'_{k}-i}\sum_{s=m'_{j}-b}^{m_{j}}\sum_{r=m'_{l}-c}^{m_{l}}c_{n,k,p,i}\binom{m'_{k}-t}{i}K_{j,b,s}K_{l,c,r}\tilde K_{j+l,s+r}\langle \Phi_{k,t},\Phi_{j+l,r+s}^*\rangle.
\end{align*}
Hence, the coefficient of $S_{j}^{(p),b}\otimes S_{l}^{(p),c}$ in $A^{\O(n)}_{1,p+1}(S_k^{(p),i})$ is zero unless $k=j+l$ and $t=s+r$ and the result follows by substituting the constants and taking into account the possible range for $t,s$ and $r$.
\endproof

\def\cprime{$'$}

\end{document}